\documentclass[11pt]{nyjm}

\usepackage{hyperref,color,amsthm,amssymb,graphicx}
\usepackage{thm-restate}
\usepackage{enumerate}
\usepackage{xcolor}
\usepackage{overpic}
\usepackage{cleveref}
\usepackage{verbatim}
\usepackage{subcaption}
\usepackage{ulem}
\usepackage[left=4cm, right=4cm, marginpar=3cm]{geometry}
\usepackage{fullpage}
\usepackage{tikz-cd}
\usepackage{MnSymbol}
\usepackage{graphicx}

\usepackage{calc}
\usepackage{graphicx,import}
\newtheorem{theorem}{Theorem}[section]

\newtheorem{lemma}[theorem]{Lemma}
\newtheorem{corollary}[theorem]{Corollary}
\newtheorem{proposition}[theorem]{Proposition}

\newtheorem*{claim*}{Claim}

\theoremstyle{definition}
\newtheorem{remark}[theorem]{Remark}
\newtheorem{definition}[theorem]{Definition}
\newtheorem{example}[theorem]{Example}
\newtheorem{exs}[theorem]{Examples}

\newcommand{\mc}[1]{\mathcal{#1}}

\newcommand{\Z}{\mathbb Z}
\newcommand{\N}{\mathbb N}
\newcommand{\R}{\mathbb R}

\newcommand{\push}{h}
\newcommand{\mpush}{x}

\newcommand{\connsumv}{\underset{v\in V(\Gamma)}{\#} }

\renewcommand{\>}{\rangle}
\newcommand{\arr}{\rightarrow}
\newcommand{\inv}{^{-1}}

\newcommand{\familya}{\mathcal B}
\newcommand{\familyb}{\mathcal C}

\DeclareMathOperator{\Aut}{Aut}

\DeclareMathOperator{\Homeo}{Homeo}
\DeclareMathOperator{\Diffeo}{Diffeo}
\DeclareMathOperator{\supp}{supp}

\newcommand{\map}{\operatorname{Map}}
\newcommand{\mapS}{\operatorname{Map}(S)}
\newcommand{\pmapS}{\operatorname{PMap}(S)}
\newcommand{\pmap}{\operatorname{PMap}}

\newcommand{\mapcS}{\operatorname{Map}_c(S)}

\def\G{{\Gamma}}

\title[Subgroups of Big Mapping Class Groups]{A New Construction of Subgroups of \\ Big Mapping Class Groups}

\author[C.R. Abbott]{Carolyn R. Abbott}
\address{Department of Mathematics, Mailstop 050, Brandeis University, 415 South Street, Waltham, MA 02453}
\email{carolynabbott@brandeis.edu}

\author[H. Hoganson]{Hannah Hoganson}
\address{Department of Mathematics, William E. Kirwan Hall, 4176 Campus Dr., University of Maryland College Park, MD 20742-4015}
\email{hoganson@umd.edu} 

\author[M. Loving]{Marissa Loving}
 \address{Department of Mathematics, 480 Lincoln Drive, Van Vleck Hall, Madison, WI 53706}
 \email{mloving2@wisc.edu}
 
\author[P. Patel]{Priyam Patel}
 \address{Department of Mathematics, John A. Widtsoe Building, University of Utah, 155 S 1400 E, Salt Lake City, UT 84112}
 \email{patelp@math.utah.edu} 

\author[R. Skipper]{Rachel Skipper}
 \address{Department of Mathematics, John A. Widtsoe Building, University of Utah, 155 S 1400 E, Salt Lake City, UT 84112}
 \email{rachel.skipper@utah.edu}

\date{\today}



\begin{document}

\begin{abstract} 
We explicitly construct new subgroups of the mapping class groups of an uncountable collection of infinite-type surfaces, including, but not limited to, free groups, Baumslag-Solitar groups, mapping class groups of other surfaces, and a large collection of wreath products. For each such subgroup $H$ and surface $S$, we show that there are countably many non-conjugate embeddings of $H$ into $\mapS$; in certain cases, there are uncountably many such embeddings. The images of each of these embeddings cannot lie in the isometry group of $S$ for any hyperbolic metric and are not contained in the closure of the compactly supported subgroup of $\mapS$. In this sense, our construction is new and does not rely on previously known techniques for constructing subgroups of mapping class groups. Notably, our embeddings of $\map(S')$ into $\map(S)$ are not induced by embeddings of $S'$ into $S$. Our main tool for all of these constructions is the utilization of special homeomorphisms of $S$ called shift maps, and more generally, multipush maps. 
\end{abstract} 

\maketitle

\section{Introduction}\label{S:intro}

A fundamental question in low-dimensional topology asks which groups can arise as subgroups of the diffeomorphism group, homeomorphism group, and mapping class group of a surface $S$, denoted by $\Homeo(S)$, $\Diffeo(S),$ and $\mapS$, respectively. One approach to producing such subgroups is to consider embeddings of finite-type subsurfaces $S'$ into an infinite-type surface $S$  that induce injections of $\map(S')$  into $\map(S)$. In this case, every subgroup of $\map(S')$ is a subgroup of $\mapS$. Another approach to this problem is to show that a particular group $G$ acts by orientation-preserving isometries on a surface $S$, which implies that $G$ can be realized as a subgroup of $\Homeo(S)$, $\Diffeo(S),$ and $\mapS$. 
However, these two classical approaches have limitations. For example, the strong Tits alternative holds for finite-type mapping class groups, meaning every subgroup of $\map(S')$ is either virtually abelian or contains a free subgroup \cite{Ivanov, McCarthy}. In addition, Aougab, Patel, and Vlamis show that only finite groups can arise as the isometry group of a hyperbolic metric on $S$ whenever $S$ contains a \emph{non-displaceable subsurface} (see \cite[Lemma 4.2]{APV2}). They also show that no uncountable group can be obtained as the isometry group of a hyperbolic metric on any infinite-type surface. These observations indicate that in order to fully understand the algebraic structure of big mapping class groups, we  need other constructions of subgroups in $\mapS$, and we also need to understand the many different ways that a particular subgroup can embed in $\mapS$. This is precisely the goal of this paper. 

To streamline the statements of our results below, we construct an uncountable collection of surfaces for which particular results  hold. The precise definitions will appear in \Cref{sec:non-conjugate}. When $\Pi$ is a \textit{distinguished surface}, we denote by $\familyb(\Pi)$ the collection of surfaces that admit a  map which acts as a shift along a countable collection of copies of $\Pi$; see \Cref{def:familyb}. For example, in the case where $\Pi$ is a torus with one boundary component, $\familyb(\Pi)$ includes the ladder surface, the connect sum of a ladder surface and any surface of genus 0, and the Cantor tree surface with exactly two ends accumulated by genus.

Our first construction produces embeddings between big mapping class groups that are not induced by embeddings of the underlying surfaces and that do not preserve the property of being compactly supported. The result is summarized as \Cref{thm:indicable} below and is a consequence of \Cref{T:otherGs}.
Recall that a group is  indicable if it admits a surjection onto $\Z$. 
 
 \begin{restatable}{theorem}{MCG}
\label{thm:indicable}
Let $\Pi$ be a distinguished surface.  If  $\map(\Pi)$ is indicable, then for any surface $S \in \familyb(\Pi)$, there exist countably many non-conjugate embeddings of $\map(\Pi)$ into $\mapS$ that are not induced by an embedding of $\Pi$ into $S$.
\end{restatable}

The above theorem is in line with a body of work dedicated to understanding and constructing homomorphisms between mapping class groups; see, for example, \cite{AramayonaLeiningerSouto, AramayonaSouto, ALM}.  There are uncountably many distinguished surfaces $\Pi$ for which $\map(\Pi)$ is indicable; see \Cref{ex:indicable}. When $\Pi$ has at least two nonplanar ends, \Cref{thm:indicable} also holds with $\pmap(\Pi)$ in place of  $\map(\Pi)$; see \Cref{cor:PureEmbedding}. 
 
\Cref{thm:indicable} also answers Question 4.75 from the AIM problem list on surfaces of infinite type \cite{AIM-PL} which asks, “Given a homomorphism $f \colon \map(S) \to \map(S')$, does $f$ preserve the notion of being compactly supported?”  Bavard, Dowdall and Rafi \cite{BVD} show that the answer is yes for surjective homomorphisms, and  Aramayona, Leininger, and McLeay \cite{ALM} give an example of two surfaces and self-maps for which the answer is no. Our results show  that there is an uncountable family of surfaces and maps for which the answer is also no.

The second main result of our paper is \Cref{cor:otherS}, which is a combination theorem for indicable subgroups of $\mapS$. We summarize its statement as \Cref{thm:main} below. The $\star$-product used in the statement 
is an interpolation between free products and direct products. The $\star$-product of two groups $G_i$ with subgroups $H_i$ is defined as 
\[ (G_1, H_1) \star (G_2, H_2) := G_1\ast G_2 /\llangle [G_1,H_2],[H_1,G_2]\rrangle.
\]
More generally, given $G_1, \ldots, G_n$ with  subgroups $H_1, \ldots, H_n$, let \[(G_1, H_1) \star \cdots \star (G_n, H_n) := G_1 \ast \cdots \ast G_n / \llangle [G_i, H_j]: i \neq j \rrangle.\]

\begin{restatable}{theorem}{Main}\label{thm:main}
    Let $G_i$ be an indicable group that embeds in $\map(S_i)$ for $i=1, \ldots, n$, where $S_i$ is a surface with exactly one boundary component.  For each i, fix a surjective map $f_i \colon G_i \to \Z$, and let $H_i$ be the kernel of  $f_i$.  The indicable group $(G_1, H_1) \star \cdots \star (G_n, H_n)$ embeds in $\mapS$ for  $S = S_\Gamma(\Pi)$, where $\Pi$ is obtained from $\#_n S_i$ by capping off $n-1$ boundary components. 
    \end{restatable}

We direct the reader to \Cref{sec:constructions} for the definition of the surface $S_\Gamma(\Pi)$, the construction of which was inspired by work of Allcock \cite{All06}. Importantly, the support of the homeomorphisms defined in our construction is not all of $S_\Gamma(\Pi)$. Consequently, we may change the topology outside the support of the homeomorphisms in any way we choose.  In this way, \Cref{thm:main} actually shows that $(G_1, H_1) \star \cdots \star (G_n, H_n)$ embeds in the mapping class group of a wide class of infinite-type surfaces. For instance, we can arrange for the edited surface to have a non-displaceable subsurface so that the subgroups we construct cannot arise from a construction using isometries. 

A key aspect of the proof of \Cref{thm:main} is a set of criteria on a collection of \emph{shift maps} (or \emph{multipush maps}) that guarantees they generate a free group (see \Cref{thm:freegroup}). Shift maps are generalizations of \emph{handleshifts}, homeomorphisms introduced by Patel and Vlamis in \cite{PatelVlamis} that have become integral to the theory of infinite-type surfaces. In particular, we augment the generators of the groups $G_i$ in the statement of the theorem with these shift maps (or multipush maps). The fact that shift maps do not lie in $\overline{\mapcS}$ implies that the subgroups we construct are also not completely contained in $\overline{\mapcS}$. The only exceptions to this are when $S$ is finite type or the Loch Ness Monster surface, in which case $\overline \mapcS = \mapS$. We avoid the technical statement of \Cref{thm:freegroup} here and direct the reader to \Cref{sec:free}.

There are a variety of indicable groups that can play the role of $G_i$ in the statement of \Cref{thm:main} (or the role of $G$ in the statement of \Cref{thm:indicable}). In particular, one can let $G_i$  be any indicable subgroup of the mapping class group of a finite-type surface with exactly one boundary component, for example, free groups, braid groups, and right-angled Artin groups. 

In the case of right-angled Artin groups, Clay, Leininger, Mangahas \cite{CLM12} and Koberda \cite{Koberda} show that every right-angled Artin group embeds as a subgroup of the mapping class group of some finite-type surface of sufficient complexity. The map $f_i: A_{\G_i} \arr \Z$ which sends each generator of the right-angled Artin group $A_{\G_i}$ with defining graph $\Gamma_i$ to $1$ is a homomorphism, so  $A_{\G_i}$ is indicable. The kernel $H_i$ of this map is called the \textit{Bestvina-Brady group} defined on $\G_i$ and is denoted by $BB_{\G_i}$. The pairs $(A_{\G_i}, BB_{\G_i})$ form a rich class of examples that can be used as the input for \Cref{thm:main} and are discussed in detail in \Cref{sub:RAAGs}.

We also produce new examples of indicable subgroups of big mapping class groups that can be used as input for these theorems, including solvable Baumslag-Solitar groups $BS(1,n)$ and  a large class of wreath products $G \wr H$. 
The following theorem is a particular case of \Cref{prop:GwrHgeneral} and \Cref{thm:BS1n}, which both hold for a  more general class of surfaces. We make the statement below to avoid technicalities. 

\begin{theorem}\label{intro-subgroups-statement}
If $S$ is a Cantor tree surface, then solvable Baumslag-Solitar groups $BS(1,n)$  and wreath products $\Z^n \wr \Z$ for any $n\geq 1$ arise as subgroups of $\mapS$.
\end{theorem}

Note that  solvable Baumslag-Solitar and $\Z^n\wr \Z$  cannot embed in the mapping class group of any finite-type surfaces. Our theorem gives the first construction of these groups (for $n>1$) in the mapping class groups of infinite-type surfaces. Lanier--Loving construct $\Z \wr \Z$ as a subgroup of the mapping class group of any infinite-type surface without boundary \cite{LanierLoving}.  

\subsection*{Outline} \Cref{sec:prelim} contains preliminaries on infinite-type surfaces, mapping class groups, and  shift and multipush maps. \Cref{sec:constructions} gives a construction of surfaces based on Schreier graphs and describes how to obtain non-conjugate embeddings of subgroups generated by either shift or multipush maps. Our constructions of specific subgroups of big mapping class groups begins in \Cref{sec:subgroupconstructions}, where we build embeddings  of free groups, wreath products, and solvable Baumslag--Solitar groups into big mapping class groups. In \Cref{sec:indicable} we prove \Cref{thm:indicable}. Finally, in \Cref{sec:combo} we prove and discuss applications of the combination theorem (\Cref{thm:main}).

\subsection*{Acknowledgements}
The authors would like to thank Women in Groups, Geometry, and Dynamics (WiGGD) for facilitating this collaboration, which was supported by NSF DMS--1552234, DMS--1651963, and  DMS--1848346.  The authors also thank Mladen Bestvina and Robbie Lyman for helpful conversations, as well as George Domat for productive discussions about surfaces with indicable mapping class groups and an anonymous referee providing comments which significantly improved the paper. 

In addition, the authors acknowledge support from NSF grants DMS--1803368, DMS--2106906, and DMS-2340341 (Abbott), DMS--1906095 and RTG DMS--1840190 (Hoganson), DMS--1902729 and DMS--2231286 (Loving), DMS--1937969 and DMS--2046889 (Patel), and DMS--2005297 and DMS--2506840 (Skipper). Skipper was also supported by the European Research Council (ERC) under the European Union’s Horizon 2020 research and innovation program (grant agreement No.725773).

\section{Preliminaries}
\label{sec:prelim}
\subsection{Ends of surfaces} Essential to the classification of infinite-type surfaces is the notion of an end of a surface and the space of ends for an infinite-type surface $S$. 

\begin{definition}
An \textit{exiting sequence} in \( S \) is a sequence \( \{U_n\}_{n\in\N} \) of connected open subsets of \( S \) satisfying:
\begin{enumerate}
\item \( U_{n} \subset U_m \) whenever $m<n$;
\item $U_n$ is not relatively compact for any $n \in \N$, that is, the closure of $U_n$ in $S$ is not compact;
\item the boundary of \(U_n \) is compact for each \( n \in \N \); and
\item any relatively compact subset of $S$ is disjoint from all but finitely many of the $U_n$’s.
\end{enumerate}
Two exiting sequences \( \{U_n\}_{n\in\N} \) and \( \{V_n\}_{n\in\N} \) are equivalent if for every \( n \in \N \) there exists \( m \in \N \) such that \( U_m \subset V_n \) and \( V_m \subset U_n \). An \textit{end} of \( S \) is an equivalence class of exiting sequences. 
\end{definition}

The \textit{space of ends} of S, denoted by $E(S)$, is the set of ends of $S$ equipped with a natural topology for which it is totally disconnected, Hausdorff, second countable, and compact. In particular, $E(S)$ is homeomorphic to a closed subset of a Cantor set. The definition of the topology on the space of ends is not relevant to this paper and so is omitted.

Ends of $S$ can be isolated or not and can be \emph{planar}, if there exists an $i$ such that $U_i$ is homeomorphic to an open subset of the plane $\mathbb{R}^2$, or \emph{nonplanar}, if every $U_i$ has infinite genus. The set of nonplanar ends of $S$ is a closed subspace of $E(S)$ and will be denoted by $E^g(S)$; these are frequently called the \emph{ends accumulated by genus}. We have the following classification theorem of Ker\'ekj\'art\'o \cite{Kerekjarto} and Richards \cite{Richards}: 

\begin{theorem}[Classification of infinite-type surfaces]\label{thm:classification}
The homeomorphism type of an orientable, infinite-type surface $S$ is determined by the quadruple
$$(g, b, E^g(S), E(S))$$
where $g \in \Z_{\geq 0} \cup \infty$ is the genus of $S$ and $b \in \Z_{\geq 0}$ is the number of (compact) boundary components of S.
\end{theorem} 

There is a more complicated classification of infinite-type surfaces allowing for non-compact boundary components due to Prishlyak--Mischenko \cite{Prishlyak-Mischenko2007}. We  use this classification once in \Cref{sec:non-conjugate}, but in our setting, the surfaces we are comparing have precisely the same boundary, so the classification reduces to considering the triple $(g, E^g(S), E(S))$.

\subsection{Mapping class group}
The \emph{mapping class group} of $S$, denoted  $\mapS$, is the set of orientation-preserving homeomorphisms of $S$ up to isotopy that fix the boundary pointwise. The natural topology on the set of homeomorphisms of $S$ is the compact-open topology, and $\mapS$ is endowed with the induced quotient topology. Equipped with this topology, $\mapS$ is a topological group. When $S$ is a finite-type surface, this topology on $\mapS$ agrees with the discrete topology, but when $S$ is of infinite type, the two topologies are distinct. The \emph{pure mapping class group}, denoted $\pmapS$,  is the subgroup of $\mapS$ that fixes the set of ends of $S$ pointwise, and $\mapcS$ is the subgroup of compactly supported mapping classes. Note that $\overline\mapcS \leq \pmapS$. When $S$ has at most one nonplanar end, $\overline\mapcS$ is actually equal to $\pmapS$ \cite{APV1}.

\begin{definition}\label{def:IntrinsicInfinite} A mapping class $f \in \mapS$ is of \emph{intrinsically infinite type} if $f \notin \overline\mapcS$. A subgroup $H \leq \mapS$ is of intrinsically infinite type if $H$ is not completely contained in $\overline{\mapcS}$. 
\end{definition}

\noindent In this paper, all of the subgroups of $\mapS$ that we construct contain many intrinsically infinite-type homeomorphisms and, therefore, cannot be completely contained in $\overline \mapcS$, except when $S$ is finite-type or the Loch Ness Monster, in which case $\overline \mapcS=\mapS$. Recall that the Loch Ness Monster surface is the unique infinite-genus surface with one end (up to homeomorphism). 

We are particularly interested in indicable groups and various ways of embedding them in mapping class groups of infinite-type surfaces. A group $G$ is \emph{indicable} if there exists a surjective homomorphism $f\colon G \to \Z$. We show in \Cref{L:indicable} that a group $G$ is indicable if and only if there is a presentation for $G$ where the relators all have total exponent sum zero in the generators of $G$. Importantly, many of our constructions require an indicable subgroup $G$ of $\mapS$  as an input, where $S$ is a surface with exactly one boundary component. There are many examples of such groups that were mentioned in the introduction, but there are also some restrictions on what groups $G$ can arise as subgroups of mapping class groups, as is evidenced by the following lemma, which generalizes the same result from the finite-type setting \cite[Corollary~7.3]{primer}. 

\begin{lemma} [{\cite[Corollary 3]{ACCL20}}] If $S$ is an orientable surface with nonempty compact boundary, the mapping class group fixing the boundary pointwise is torsion-free.
\end{lemma} 

\subsection{Push and shift maps}
In this section, we define shift maps and push maps, which are central to all of our constructions.
A particular type of shift maps, called handle shifts, were first studied by Patel and Vlamis in \cite{PatelVlamis}. This inspired the following definition of Abbott, Miller, and Patel \cite{AMP21}. A similar definition of shift maps appears in \cite{MannRafi} and \cite{LanierLoving}.  

\begin{definition}\label{D:shift}
Let $D_\Pi$ be the surface defined by taking the strip $\mathbb{R}\times [-1,1]$, removing an open disk of radius $\frac{1}{4}$ with center $(n, 0)$ for $n \in \mathbb{Z}$, and attaching any fixed topologically nontrivial surface $\Pi$ with exactly one boundary component to the boundary of each such disk. A \emph{shift} on $D_\Pi$ is the homeomorphism that acts like a translation, sending $(x,y)$ to $(x +1, y)$ for $y \in [-1 + \epsilon, 1 - \epsilon]$ and which tapers to the identity on $\partial D_\Pi$. 
\end{definition}

Given a surface $S$ with a proper embedding of $D_\Pi$ into $S$ so that the two ends of the strip correspond to two different ends of $S$ (see \Cref{fig:WreathConstruction}), the shift on $D_\Pi$ induces a \emph{shift} on $S$, where the homeomorphism acts as the identity on the complement of $D_\Pi$. If instead, we have a proper embedding of $D_{\Pi}$ into $S$ where the two ends of the strip correspond to the same end, we call the resulting homeomorphism on $S$ a \emph{one-ended shift}. Given a shift or one-ended shift $h$ on $S$, the embedded copy of $D_\Pi$ in $S$ is called the \emph{domain} of $h$. By abuse of notation, we will sometimes say that the domain of the  shift or one-ended shift $h$ is $D_\Pi$ rather than referring to it as an embedded copy of $D_\Pi$ in $S$ (when it is clear from context to which embedded copy of $D_\Pi$ we are referring). 

\begin{remark}\label{rem:infinitetype}
If the surface $\Pi$ in \Cref{D:shift} has a nontrivial end space, then a shift or one-ended shift $h$ on $S$ with domain $D_\Pi$ is not in $\pmapS$ since there are ends of $S$ that are not fixed by $h$. Thus, $h \notin \overline \mapcS$ and is of intrinsically infinite type. On the other hand, if $h$ is a shift map and if $\Pi$ is a finite-genus surface with no planar ends, then $h$ is a power of a handle shift on $S$, and the proof of \cite[Proposition 6.3]{PatelVlamis} again tells us that  $h \notin \overline \mapcS$. However, the second conclusion does not hold when $h$ is a power of a one-ended handle shift since, in that case, it follows from work in \cite{PatelVlamis} that $h \in \overline{\mapcS}$.
\end{remark}

\begin{figure}
    \centering
    \def\svgwidth{4in}
\begingroup%
  \makeatletter%
  \providecommand\color[2][]{%
    \errmessage{(Inkscape) Color is used for the text in Inkscape, but the package 'color.sty' is not loaded}%
    \renewcommand\color[2][]{}%
  }%
  \providecommand\transparent[1]{%
    \errmessage{(Inkscape) Transparency is used (non-zero) for the text in Inkscape, but the package 'transparent.sty' is not loaded}%
    \renewcommand\transparent[1]{}%
  }%
  \providecommand\rotatebox[2]{#2}%
  \newcommand*\fsize{\dimexpr\f@size pt\relax}%
  \newcommand*\lineheight[1]{\fontsize{\fsize}{#1\fsize}\selectfont}%
  \ifx\svgwidth\undefined%
    \setlength{\unitlength}{403.20001004bp}%
    \ifx\svgscale\undefined%
      \relax%
    \else%
      \setlength{\unitlength}{\unitlength * \real{\svgscale}}%
    \fi%
  \else%
    \setlength{\unitlength}{\svgwidth}%
  \fi%
  \global\let\svgwidth\undefined%
  \global\let\svgscale\undefined%
  \makeatother%
  \begin{picture}(1,0.36142385)%
    \lineheight{1}%
    \setlength\tabcolsep{0pt}%
    \put(0,0){\includegraphics[width=\unitlength,page=1]{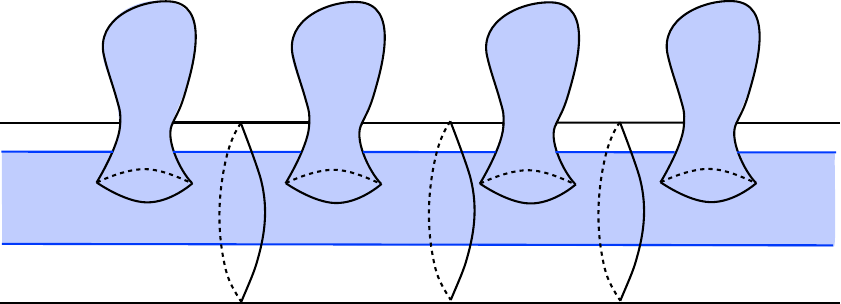}}%
    \put(0.16072293,0.30803801){\makebox(0,0)[lt]{\lineheight{1.25}\smash{\begin{tabular}[t]{l}$\Pi$\end{tabular}}}}%
    \put(0.38910431,0.3052529){\makebox(0,0)[lt]{\lineheight{1.25}\smash{\begin{tabular}[t]{l}$\Pi$\end{tabular}}}}%
    \put(0.61400421,0.3052529){\makebox(0,0)[lt]{\lineheight{1.25}\smash{\begin{tabular}[t]{l}$\Pi$\end{tabular}}}}%
    \put(0.82915608,0.31151943){\makebox(0,0)[lt]{\lineheight{1.25}\smash{\begin{tabular}[t]{l}$\Pi$\end{tabular}}}}%
    \put(0.93847282,0.11864862){\makebox(0,0)[lt]{\lineheight{1.25}\smash{\begin{tabular}[t]{l}$D_\Pi$\end{tabular}}}}%
  \end{picture}%
\endgroup%

    \caption{A surface $S$ that admits a shift whose domain is an embedded copy of $D_\Pi$.}
    \label{fig:WreathConstruction}
\end{figure}

We now use the construction of shift maps to introduce \emph{finite shifts}, which will be used in \Cref{sec:wreath} to construct certain wreath products. These are constructed in a completely analogous way, starting with an annulus instead of a biinfinite strip.

\begin{definition} 
Let $A_\Pi$ be a surface defined by taking the annulus \[([0,n]/0\sim n) \times [-1, 1],\] removing an open disk of radius $\frac{1}{4}$ centered at the integer points, and attaching any fixed topologically nontrivial surface $\Pi$ with exactly one boundary component to the boundary of each disk. A \emph{finite shift} on $A_\Pi$ is the homeomorphism that acts like a translation, sending $(x,y)$ to $(x +1, y)$ (modulo n) for $y \in [-1 + \epsilon, 1 - \epsilon]$ and which tapers to the identity on $\partial A_\Pi$. Given a surface $S$ with a proper embedding of $A_\Pi$ into $S$, the finite shift on $A_\Pi$ induces a \emph{finite shift} on $S$, where the homeomorphism acts as the identity on the complement of $A_\Pi$.  We call the embedded copy of $A_\Pi$ the \emph{domain} of the finite shift.
\end{definition}

\begin{definition}
A \emph{push} is any map that is  a finite shift, a one-ended shift, or a shift map. 
\end{definition}

In \Cref{sec:constructions}, we will introduce the notion of a \emph{multipush}, which is roughly a collection of push maps with disjoint supports, once we have developed some further notation and language.

\section{Surfaces from graphs and non-conjugate embeddings}\label{sec:constructions}

In this section, we begin by constructing a broad class of surfaces using an underlying graph. We then introduce a type of homeomorphism called a \emph{multipush} and show that these maps can be utilized to produce infinitely many non-conjugate embeddings of certain groups into mapping class groups.

\subsection{A construction of surfaces}\label{sec:surfaces}

The basic building block for this construction is a $d$--holed sphere. The following definition of seams restricts to the normal notion of seams for a 3--holed sphere, i.e., a pair of pants. 

\begin{definition}\label{defn:seams}
A set of \emph{seams} on a $d$--holed sphere is a collection of $d$ disjointly embedded arcs such that each boundary component of the sphere intersects  exactly two components of the seams at two distinct points and such that the collection of seams divides the sphere into two components. Call one component the \emph{front side} and the other component the \emph{back side}. These conditions imply that each component is homeomorphic to a disk.
\end{definition}

Starting from any graph $\Gamma$ with a countable vertex set and any surface $\Pi$ with exactly one boundary component, we describe a procedure for building a surface $S_\Gamma(\Pi)$. 
This mirrors a construction of Allcock using the Cayley graph of a given group $G$ \cite{All06}. 

For each vertex $v$ of valence $d$, start with a $d$--holed sphere. Remove a disk on the interior of the front side, and attach the surface $\Pi$ along the boundary component. Call the resulting surface the \emph{vertex surface for $v$}, which we denote by $V_v$, and let $\Pi_v$ be the copy of $\Pi$ on $V_v$. For each edge of the graph, define the edge surface $E$ to be the $2$--holed sphere with seams; topologically this is an annulus.

Whenever $u$ and $v$ are  vertices of $\Gamma$ connected by an edge,  connect the vertex surfaces $V_u$ and $V_v$ with an edge surface $E(u,v)$ by gluing  one boundary component of the edge surface to a boundary component of $V_u$ and the other boundary component of the edge surface to a boundary component of $V_v$ so that the gluing is \emph{compatible} in the following sense: the union of the seams separates $S_\Gamma$ into two disjoint connected components, the \emph{front} and the \emph{back}, containing the front and, respectively, the back of each vertex and edge surface.  Call the resulting surface $S_\Gamma(\Pi)$.  See \Cref{fig:S_Gamma(Pi)} for an example.   Notice that the assumption that the vertex set $V(\Gamma)$ of $\Gamma$ is countable is necessary for this construction to yield a surface.  In particular, if $V(\Gamma)$ is uncountable, then $S_\Gamma(\Pi)$ is not second countable and therefore cannot be a surface. 

\begin{figure}
    \centering
    \def\svgwidth{3.75in}
    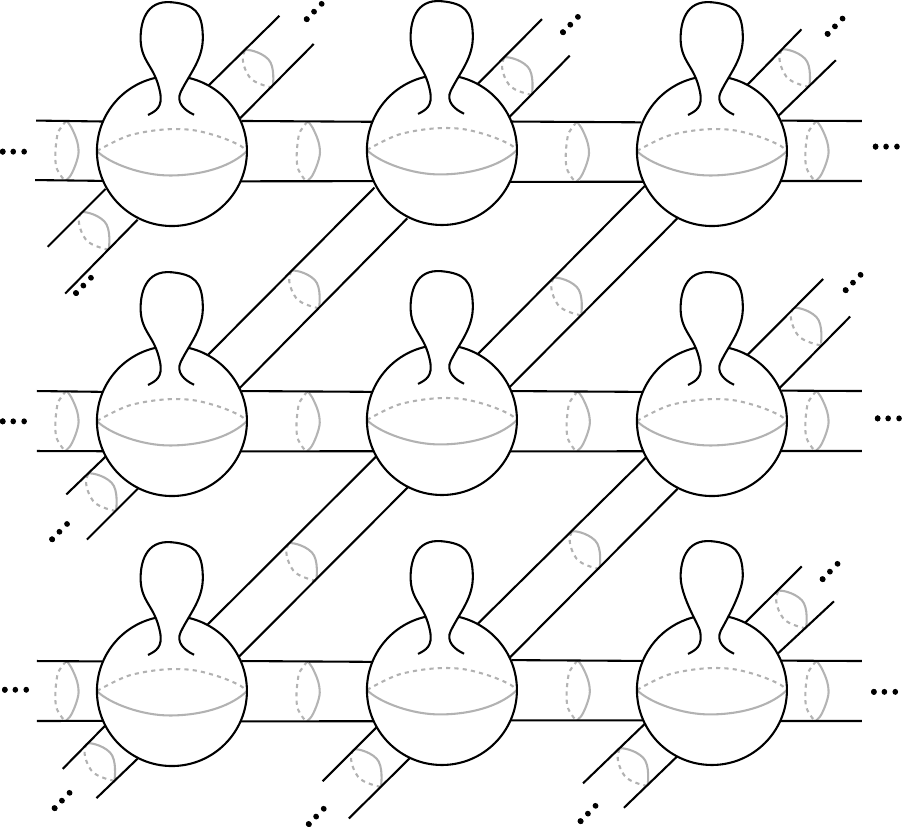
    \caption{An example of the surface $S_\Gamma(\Pi)$ where the graph $\Gamma$ is the Cayley graph of the group $\mathbb Z^2=\langle a,b : [a,b]\rangle$.}
    \label{fig:S_Gamma(Pi)}
\end{figure}

We  also define a more general class of surfaces constructed by editing the back of $S_\Gamma(\Pi)$ as follows.  As above, fix a graph $\Gamma$  with a countable vertex set and a surface with one boundary component $\Pi$, and let $S=S_\Gamma(\Pi)$.  Given any  collection of surfaces $\{\Omega_v\}_{v \in V(\Gamma)}$, only finitely many of which have boundary, we form the surface $S \connsumv \Omega_v$ as follows. For each $v\in V(\Gamma)$, take the connect sum of $V_v$ and the corresponding $\Omega_v$. It is
helpful to assume that the connect sum is done on the back
of $V_v$, since we will perform certain homeomorphisms on the front of $S$ later in the paper. We note that if every $\Omega_v$ is a sphere, then $S\connsumv \Omega_v$ is homeomorphic to $S$.  On the other hand, by choosing the $\Omega_v$ to be more topologically complex, the homeomorphism type of the resulting surface differs from $S$ in either the genus or the space of ends. Thus, even for a fixed surface $\Pi$, this construction yields a large family of surfaces, formed by varying the topological type of the  $\Omega_v$.

The underlying graph $\Gamma$ used to build $S_\Gamma(\Pi)$ throughout this paper will often be a Schreier graph, which is defined as follows. Let $G$ be a finitely generated group, $H$ a subgroup of $G$, and $T$  a finite generating set for $G$. The Schreier graph $\Gamma(G,T,H)$ is the graph whose vertices are the right cosets of $H$ and in which, for each coset $Hg$ and each $s\in T$, there is an edge from $Hg$ to $Hgs$ labeled by $s$. If $Hg = Hgs$, there is a loop labeled by $s$ at the vertex corresponding to $Hg$. Our assumption on the finiteness of $T$ ensures that $\Gamma(G,T,H)$ has a countable vertex set. When $\Gamma$ is a Schreier graph,  let $\Pi_{Hg}$ be the copy of $\Pi$ on the vertex surface corresponding to the coset $Hg$.  In the special case when $H=\{1\}$,  the Schreier graph $\Gamma(G,T,\{1\})$ is simply the Cayley graph of $G$ with respect to the generating set $T$, which we denote by $\Gamma(G,T)$.

\begin{definition}
Let $\Gamma$ be a Schreier graph for a triple $(G, T, H)$.   A \textit{Schreier surface associated to $(G,T,H)$} is a surface $S=S_\Gamma(\Pi)\connsumv \Omega_v$ where   $\Pi$ has exactly one boundary component and is not a disk, and $\{\Omega_v\}$ is any collection of surfaces,  only finitely many of which have boundary.    
\end{definition}

Using the more general class of Schreier graphs, rather than just Cayley graphs, to construct surfaces in this fashion broadens the class of surfaces our results apply to. For example, when $\Pi$ is compact, the surface $S_{\Gamma}(\Pi)$ will have the same end space as the graph $\Gamma$. A Cayley graph for a finitely generated group has $1,$ $2$ or a Cantor set of ends. On the other hand, there are many more possibilities for a  Schreier graph; any regular graph with even degree can be realized as a Schreier graph \cite{Gro77, Lub95}. Thus, there are Schreier graphs with any finite number of ends, or end spaces isomorphic to $\N \cup \{\infty\}$ or $\{-\infty\}\cup \Z \cup \{\infty\}$. This implies that there are Schreier surfaces with these end spaces as well.  See \Cref{fig:SchreierExs} for two examples of Schreier graphs that cannot be realized as Cayley graphs.

\begin{figure}
    \centering
    \includegraphics[width=5in]{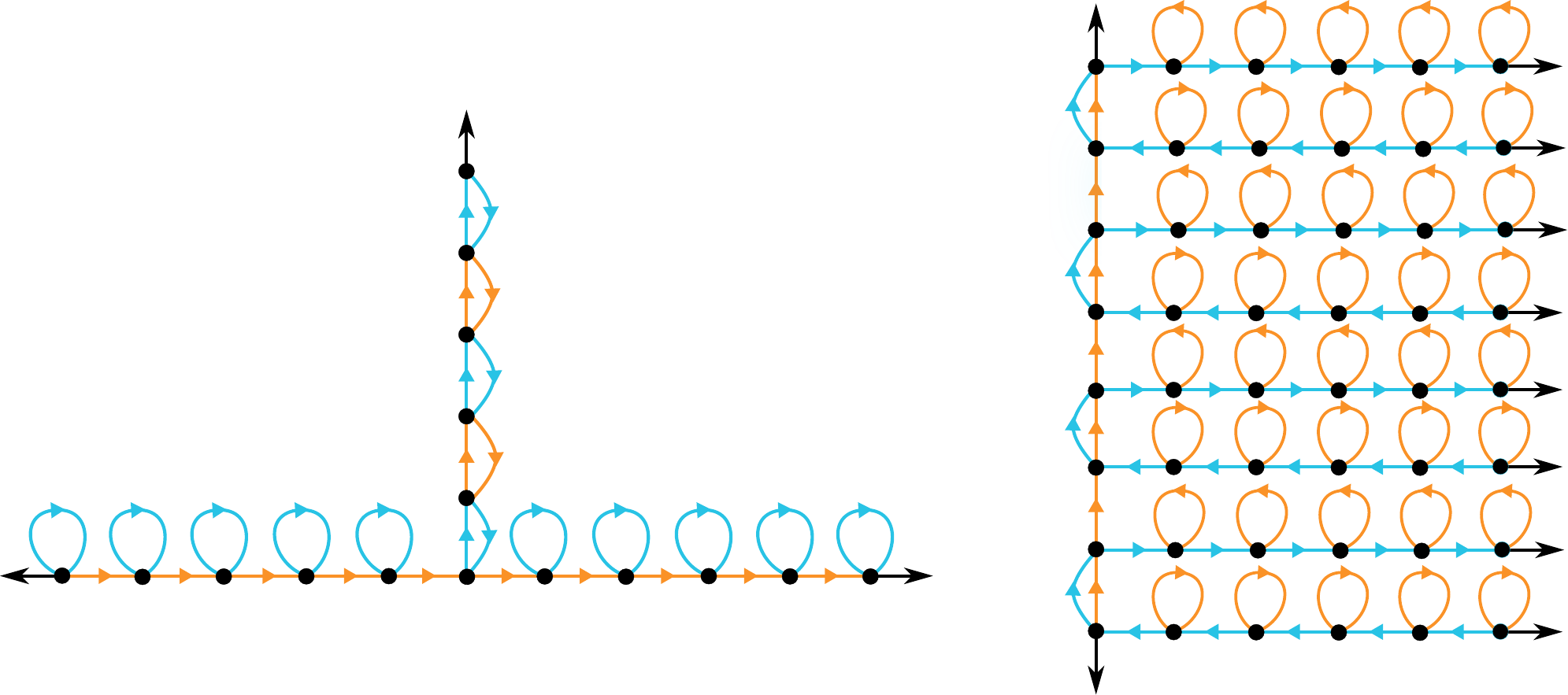}
    \caption{Each graph can be realized as a Schreier graph but not a Cayley graph.  The graph on the left has 3 ends, and the graph on the right has end space homeomorphic to the $2$-point compactification of $\Z$.}
    \label{fig:SchreierExs}
\end{figure}

We now move to defining a homeomorphism called a \emph{multipush} on a Schreier surface.

\begin{definition}\label{D:mpush} Let  $\Gamma = \Gamma(G,T, H)$ be a Schreier graph.  Fix a surface $\Pi$ with exactly one boundary component. Let $S= S_\Gamma(\Pi)\connsumv\Omega_v$ be a Schreier surface.

For each $s\in T$, we construct a collection of push maps whose support corresponds to connected components of the subgraph of $\Gamma$ which includes only edges labeled by $s$ (see \Cref{fig:multipush}). Fix a transversal $\mathcal T$ for the set of double cosets $\{Hg\langle s\rangle\mid g\in G\}$, so that $\mc T$ contains exactly one element from each double coset in the set.  In the case $H=\{1\}$, that is, when $\Gamma$ is a Cayley graph, the set $\mc T$ is simply a transversal for (left cosets of) $\langle s\rangle$.
 For each element $t$ in the transversal, we define a push $\push_{t\langle s \rangle }$ which maps $\Pi_{Hts^i}$ to $\Pi_{Hts^{i+1}}$. The support of $\push_{t\langle s \rangle }$ is contained in the front of \[\left(\bigcup_{i\in \Z}V_{ Hts^i}\right) \,\bigcup\, \left(\bigcup_{i\in \Z}E(Hts^i, Hts^{i+1}) \right).\] Recall that $V_{Hts^i}$ is the vertex surface associated to the vertex $Hts^i$ and $E(Hts^i, Hts^{i+1})$ is the edge surface associated to the edge $(Hts^i, Hts^{i+1})$ for each $i \in \mathbb Z$. 
 This support corresponds to a connected component of $\Gamma$ with all edges labeled by $s$; see \Cref{fig:multipush}.
 The \emph{multipush} $\mpush_s$ associated to $ s$ is the element of $\mapS$ that acts simultaneously as the pushes $\push_{t \langle s \rangle}$ for each $t\in \mathcal T$.  We let $D_s$ denote the domain of the multipush $\mpush_s$. If  $h_{t\langle s\rangle}$ is not a finite shift for any $t\in T$,  we say $x_s$ is an \textit{infinite} multipush. 
 \end{definition}

\begin{figure}
    \centering
    \def\svgwidth{4in}
    \import{images/}{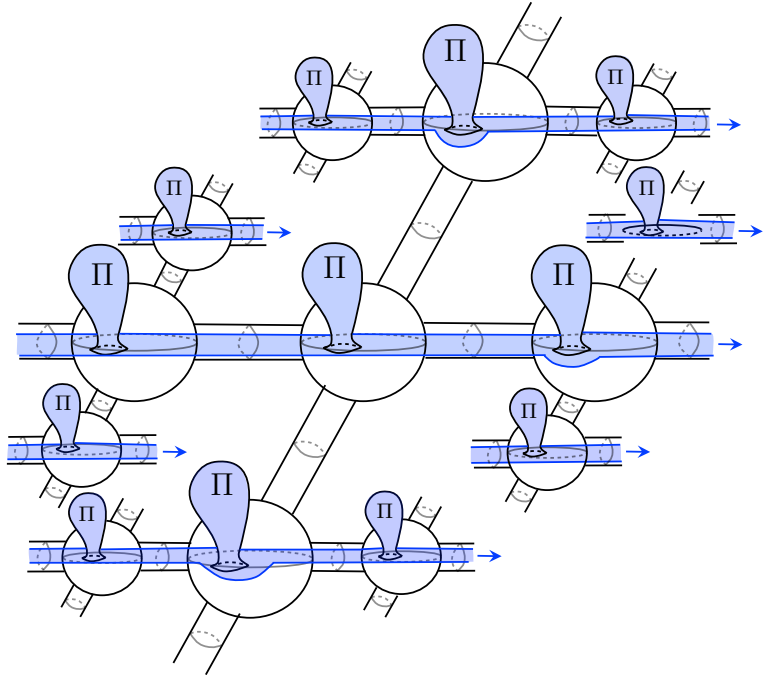}
    \caption{A portion of the domain $D_a$ (in blue) of the multipush $\mpush_a$ on the surface $S_\Gamma(\Pi)$, where $\Gamma$ is the Cayley graph $\Gamma=\Gamma(\mathbb F_2,\{a,b\})$ for $\mathbb F_2=\langle a,b\rangle$.}
    \label{fig:multipush}
\end{figure}

Since supports of the pushes $h_{t\<s\>}$ are disjoint, the multipush $x_s$ is a well-defined homeomorphism of the surface.  Note that if $Hg=Hgs$, then the edge $E(Hgs^i,Hgs^{i+1})$ in the support of  $h_{g\langle s\rangle}$ is a loop. In particular, finite pushes can occur as part of a multipush. In \Cref{fig:multipush}, all pushes are infinite, but the multipush with the orange domain shown in \Cref{fig:multipush-finitepush} has both a finite and infinite push.

\begin{remark}
We emphasize that multipushes are \textit{not} induced by an action on the graph $\Gamma$, even when $\Gamma$ is a Cayley graph.  The construction simply uses the labeling of the vertices of $\Gamma$ to define the homeomorphism $\mpush_s$.
\end{remark}

\begin{figure}
    \centering
    \def\svgwidth{3in}
\begingroup%
  \makeatletter%
  \providecommand\color[2][]{%
    \errmessage{(Inkscape) Color is used for the text in Inkscape, but the package 'color.sty' is not loaded}%
    \renewcommand\color[2][]{}%
  }%
  \providecommand\transparent[1]{%
    \errmessage{(Inkscape) Transparency is used (non-zero) for the text in Inkscape, but the package 'transparent.sty' is not loaded}%
    \renewcommand\transparent[1]{}%
  }%
  \providecommand\rotatebox[2]{#2}%
  \newcommand*\fsize{\dimexpr\f@size pt\relax}%
  \newcommand*\lineheight[1]{\fontsize{\fsize}{#1\fsize}\selectfont}%
  \ifx\svgwidth\undefined%
    \setlength{\unitlength}{404.09613059bp}%
    \ifx\svgscale\undefined%
      \relax%
    \else%
      \setlength{\unitlength}{\unitlength * \real{\svgscale}}%
    \fi%
  \else%
    \setlength{\unitlength}{\svgwidth}%
  \fi%
  \global\let\svgwidth\undefined%
  \global\let\svgscale\undefined%
  \makeatother%
  \begin{picture}(1,1.0022915)%
    \lineheight{1}%
    \setlength\tabcolsep{0pt}%
    \put(0,0){\includegraphics[width=\unitlength,page=1]{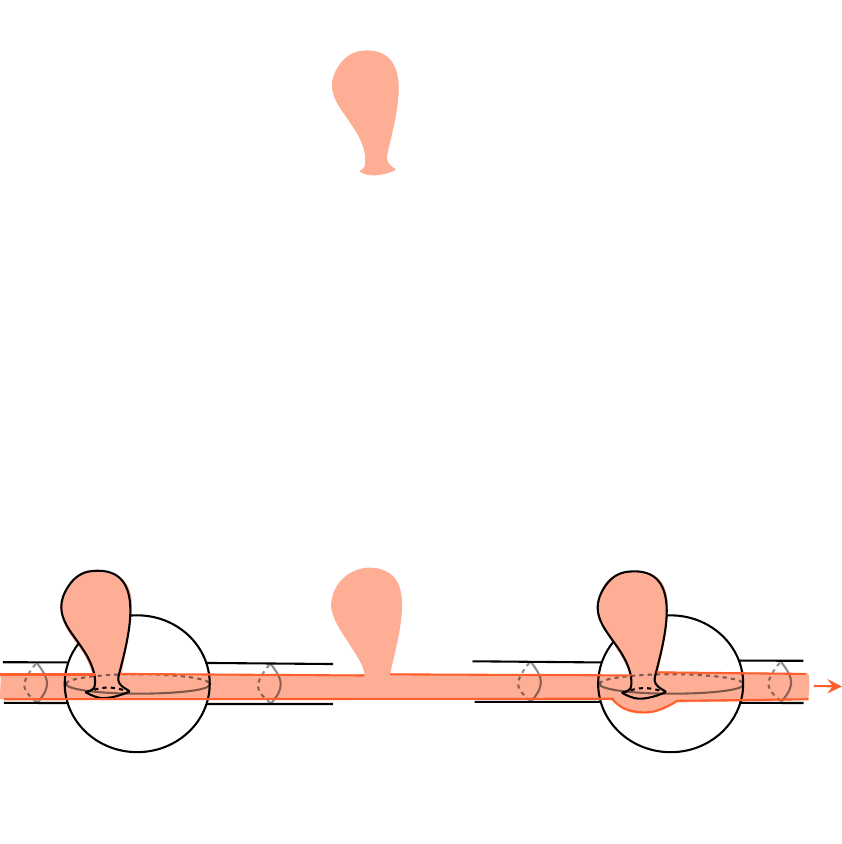}}%
    \put(0.10316864,0.26841667){\makebox(0,0)[lt]{\lineheight{1.25}\smash{\begin{tabular}[t]{l}$\Pi$\end{tabular}}}}%
    \put(0,0){\includegraphics[width=\unitlength,page=2]{Multipush-finitepush.pdf}}%
    \put(0.42300997,0.27803588){\makebox(0,0)[lt]{\lineheight{1.25}\smash{\begin{tabular}[t]{l}$\Pi$\end{tabular}}}}%
    \put(0.73804154,0.26841659){\makebox(0,0)[lt]{\lineheight{1.25}\smash{\begin{tabular}[t]{l}$\Pi$\end{tabular}}}}%
    \put(0,0){\includegraphics[width=\unitlength,page=3]{Multipush-finitepush.pdf}}%
    \put(0.42300997,0.89796257){\makebox(0,0)[lt]{\lineheight{1.25}\smash{\begin{tabular}[t]{l}$\Pi$\end{tabular}}}}%
    \put(0,0){\includegraphics[width=\unitlength,page=4]{Multipush-finitepush.pdf}}%
    \put(0.42300997,0.58915843){\makebox(0,0)[lt]{\lineheight{1.25}\smash{\begin{tabular}[t]{l}$\Pi$\end{tabular}}}}%
  \end{picture}%
\endgroup%

    \caption{A multipush on a surface corresponding to the Schreier graph on the left in \Cref{fig:SchreierExs} that contains both a finite and infinite push. The domain of the multipush is highlighted in orange.}
    \label{fig:multipush-finitepush}
\end{figure}

\subsection{Non-conjugate embeddings}\label{sec:non-conjugate}

Given a shift map $h$ corresponding to an embedding of $D_\Pi$ into a surface $S$, one can define a new and distinct shift map $h'$ on $S$ by omitting some of the surfaces $\Pi_i$ from the  domain of $h$, so long as infinitely many remain. This gives another embedding of $D_\Pi$ into $S$. See \Cref{fig:RemoveDomain}.  Since there are uncountably many infinite subsets of $\Z$, we can construct uncountably many distinct embeddings of $D_\Pi$ into $S$, and thus uncountably many distinct shift maps on $S$, in this way. The same argument goes through for one-ended shifts as well. 
 Similarly, one infinite multipush $x_s$ associated to a generator $s \in T$ on a surface $S= S_\Gamma(\Pi)$ can be used to produce uncountably many distinct domains for multipushes associated to $s$ by simply omitting some of the copies of $\Pi$ from the domains of $\mpush_{s}$ for all $s \in T$.  

\begin{figure}
    \centering
    \def\svgwidth{3in}
    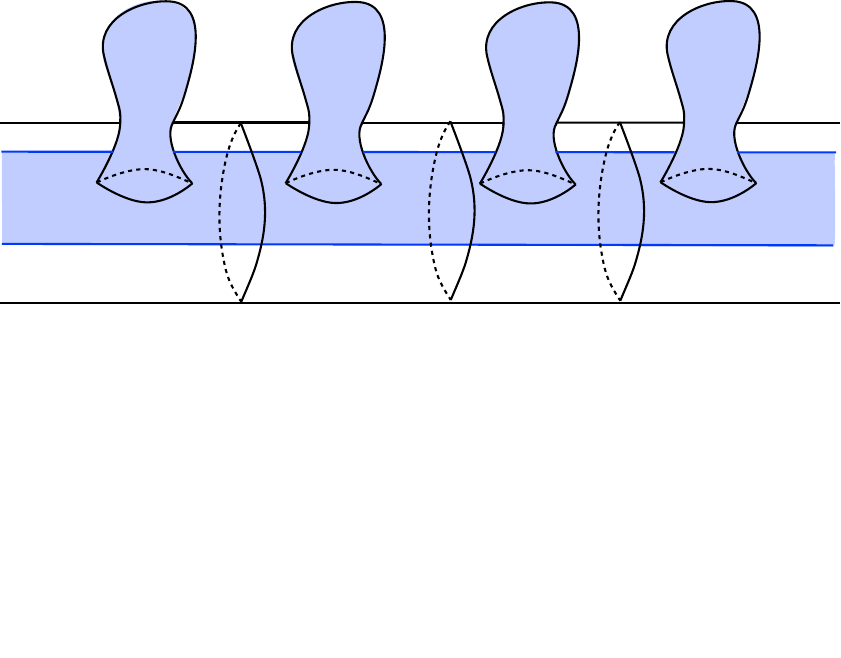
    \caption{Two different embeddings of $D_{\Pi}$. The figure at the top corresponds to a shift $h$ and the embedding shown at the bottom corresponds to a new and distinct shift $h'$ obtained by leaving two copies of $\Pi$ out of the domain of $h$.}
    \label{fig:RemoveDomain}
\end{figure}

In many cases, these distinct domains give rise to isomorphic but non-conjugate subgroups of $\mapS$.  First, consider the case of shift maps.
Given a shift map $h$ on $S$, we define a new shift map $h'$ on $S$ by removing some copies of $\Pi$ from the domain of the shift. The groups $\langle h \rangle$ and $\langle h' \rangle$ are isomorphic subgroups of $\mapS$.  If they were conjugate, not only would $\supp(h)$ and $\supp(h')$ be homeomorphic, but their complements  $S \setminus \supp(h)$ and $S \setminus \supp(h')$ would also be homeomorphic. There are many surfaces for which this latter condition fails. For example, let $\Pi$ be a handle, i.e. a torus with one boundary component, and let $S=S_{\G}(\Pi)$ for $\G(\Z,\{1\})$, i.e. the ladder surface (see \Cref{fig:RemoveDomain}). In this case, $h=x_1$ is a shift map and $S\setminus\supp(h)$ has genus zero, while $S\setminus \supp(h')$ has nonzero genus coming from the copies of $\Pi$ that were removed from the domain of $h$. Therefore, the embeddings of $\Z$ as $\<h\>$ and $\<h'\>$ are non-conjugate in $\mapS$.  This example can be generalized by letting $\Pi$ be any surface with a countable end space and one boundary component, so long as removing copies of $\Pi$ from the domain of the shift map produces non-homeomorphic subsurfaces $S\setminus \supp(h)$ and $S\setminus\supp(h')$.  This motivates the following definition.

\begin{definition}\label{def:familyb}
 A \textit{distinguished surface} is a surface $\Pi$ with exactly one boundary component, satisfying at least one of the following: \begin{enumerate}
     \item $\Pi$ has finite genus,  
     \item $E(\Pi)$ consists of finitely many planar ends, or 
     \item $E(\Pi)$ consists of finitely many nonplanar ends.
 \end{enumerate}  For each distinguished surface $\Pi$, let $\familyb(\Pi)$ be the collection of surfaces $S$ that admit an embedding of $D_\Pi$ such that the following holds. If $\Pi$ satisfies (1), then $S \setminus D_\Pi$ has finite (possibly zero) genus. If $\Pi$ satisfies (2) or (3), then $S\setminus D_\Pi$ has finitely many planar or nonplanar ends, respectively. If $\Pi$ falls into more than one of the above categories, then $\familyb(\Pi)$ should consist of surfaces that satisfy either of the conditions on $S\setminus D_{\Pi}.$
\end{definition}

If $\Pi$ is a distinguished surface and $S\in\familyb(\Pi)$, then $S$ admits a shift $h$ with domain $D_\Pi$.  If $h'$ is another shift on $S$ whose domain is embedded by omitting finitely many copies of $\Pi$ from $D_\Pi$, then each of the three  conditions on $\Pi$ ensures that $S\setminus\supp(h)$ and  $S\setminus\supp(h')$ are not homeomorphic.  In particular, $S\setminus\supp(h)$ and  $S\setminus\supp(h')$ will have different genus or will contain a different number of planar or nonplanar ends.   Similarly, if $h'$ and $h''$ are obtained from $h$ by omitting different (finite) numbers of copies of $\Pi$ from $D_\Pi$, then the complements of their supports are not homeomorphic.

The collection $\familyb(\Pi)$ for a distinguished surface $\Pi$ is uncountable.  To see this, suppose $\Pi$ has finite genus or $E(\Pi)$ consists of finitely many nonplanar ends.  Then  $S$ can be any surface such that $S\setminus D_\Pi$ has only planar ends.  On the other hand, if $E(\Pi)$ consists of finitely many planar ends, then $S$ can be any surface so that $S\setminus D_\Pi$ has no planar ends. In either case, there are uncountably many such $S$. 

The definition of a distinguished surface $\Pi$ and the collection of surfaces $\familyb(\Pi)$ ensure that we can ``count" the number of copies of $\Pi$ that have been removed from the domain of a shift, thus producing non-conjugate embeddings.  We could expand the definition of a distinguished surface and the collection $\familyb(\Pi)$ to encompass a larger family of surfaces for which this is possible, but we choose the streamlined definition above for simplicity, while still demonstrating that our results hold for a broad class of surfaces.

We have shown that there are countably many non-conjugate infinite cyclic subgroups in $\mapS$.  In \Cref{sec:indicable}, we will use these different embeddings of $\mathbb Z$ to construct non-conjugate embeddings of indicable subgroups into $\mapS$ (\Cref{thm:indicable}). The following lemma summarizes the discussion above.

\begin{lemma}\label{lem:nonconjugate2}
    Let $S$ be any surface in the uncountable collection $\familyb(\Pi)$ for a distinguished surface $\Pi$. There exist countably many non-conjugate embeddings of the subgroup generated by the shift map on $S$ with domain $D_\Pi$ into $\mapS$. 
\end{lemma}

We now turn our attention to constructing non-conjugate embeddings of subgroups generated by multipushes. Let $S=S_\Gamma(\Pi)$ be infinite-type, and let $x_s$ be the multipush defined by $s\in T$.  In the same way as for a shift map, by omitting copies of $\Pi$ from the domain of $x_s$ so that the complements of the supports are not homeomorphic, we obtain a non-conjugate embedding of $\langle x_s \rangle$ in $\mapS$.  

If several multipushes  $x_s$ for $s\in T$ have common copies of $\Pi$ in their supports, such as in \Cref{fig:multipushes}, more care needs to be taken.  It is possible to remove copies of $\Pi$ from the domains of all  the multipushes to obtain new multipushes $x_s'$  in such a way that $\langle x_s\mid s\in T\rangle\cong \langle x_s'\mid s\in T \rangle$ and so that the complements of the supports of the subgroups are not homeomorphic.  One way to formalize this is to consider the surface $S_m=S \connsumv \Omega_v$ where exactly $m$ of the $\Omega_v$ are homeomorphic to $\Pi$ with the boundary component capped off, and the remainder of the $\Omega_v$ are spheres. By the classification of surfaces, the surfaces $S$ and $S_m$ are homeomorphic, and this homeomorphism induces an isomorphism of mapping class groups $\mapS\cong \map(S_m)$.  Let $x_s^{(m)}$ be the multipush on the surface $S_m$ defined by $s\in T$.  Notice that $G=\langle x_s \mid s\in T\rangle$ is isomorphic to $\langle x_s^{(m)} \mid s\in T \rangle$ because they are generated by multipushes with the same supports $\pi_1$-embedded into different surfaces.  
Let $G_m\leq \mapS$ be the image of $\langle x_s^{(m)} \mid s\in T\rangle\leq \map(S_m)$ under the isomorphism of mapping class groups, so that $G\cong G_m$.  
 By construction, there are $m$ copies of $\Pi$ that are not in the support of $G_m$, while all copies of $\Pi$ are in the support of $G$, and so $G$ and $G_m$ are not conjugate.   Similarly, whenever $m\neq n$, the groups $G_m$ and $G_n$ are isomorphic and non-conjugate.  

For the remainder of the paper, when we say that we remove copies of $\Pi$ from the supports of  multipushes, we will mean that we do so in the above manner, so that the resulting groups are isomorphic.

If the end space of $\Gamma$ contains a Cantor set, then there are  uncountably many non-conjugate copies of $G$ in $\mapS$. To see this, 
use the procedure above to edit the domains of the multipush maps by removing a collection of copies of $\Pi$ that accumulate onto a closed subset of the Cantor set of ends of $S=S_\Gamma(\Pi)$.  By removing copies of $\Pi$ that accumulate onto non-homeomorphic closed subsets of the Cantor set, we obtain a non-conjugate embedding of $G$ into $\mapS$.

Above, we assumed that $S=S_\Gamma(\Pi)$.  However, the argument  applies more broadly.  For example, if $S= S_\Gamma(\Pi)\connsumv\Omega_v$ and each $\Omega_v$ has only planar ends and $\Pi$ has nonzero finite genus, then  adding two finite collections of handles of differing sizes to some $\Omega_v$ still results in the complements of the domains being non-homeomorphic subspaces.  More generally, we could let $\Pi$ be any surface with a countable end space and one boundary component (of which there are uncountably many), so long as removing two finite collections of $\Pi$ of differing cardinalities still results in the complement subsurfaces being non-homeomorphic.  This observation leads to the definition of the following family of surfaces.

\begin{definition} \label{def:familya}
    Let $\familya$ be the collection of Schreier surfaces $S= S_\Gamma(\Pi)\connsumv\Omega_v$ such that $S_\Gamma(\Pi)$ is infinite-type, $\Pi$ has a countable end space, and the surfaces $\Omega_v$ are compatible with $\Pi$ in following sense: 
    Let $Y= \bigcup_{s \in T} \supp(x_s)$, and $Y'=\bigcup_{s \in T} \supp(x_{s}')$, $Y''=\bigcup_{s \in T} \supp(x_{s}'')$, where $x_{s}'$ is obtained by moving $m$ copies of $\Pi$ out of the domain of  $x_s$ and $x_{s}''$ is obtained by moving $n$ copies of $\Pi$ out of the domain of  $x_s$, with $m\neq n$. Then $S\setminus Y'$ and $S\setminus Y''$ are non-homeomorphic in $S$. 
    
   Let $\mc B_\infty$ be the subset of $\mc B$ consisting of those Schreier surfaces built from infinite-type surfaces $\Pi$. 
\end{definition}

The above discussion demonstrates that $\familya$ is uncountable and proves the following lemma. 

\begin{lemma}\label{lem:nonconjugate}
    Let $S$ be any surface in the uncountable collection $\familya$. Letting $G$ be the subgroup of $\mapS$ generated by the multipush maps $x_s$ on $S$ for $s\in T$, there exist countably many non-conjugate copies of $G$ in $\mapS$. If the end space of $\Gamma$ contains a Cantor set, then there are  uncountably many non-conjugate copies of $G$ in $\mapS$.
\end{lemma}


\subsection{Non-isometric embeddings}\label{sec:non-isometric}
Throughout the paper, all constructions of subgroups will utilize push and multipush maps.  If the complement of the domain of a (multi)push is not simply connected, then the map cannot act as an isometry for any hyperbolic metric on $S$.   We can use the collection of subsurfaces $\{\Omega_v\}$  from the construction of a Schreier surface $S$ to ensure this condition holds, and so all of our constructions can produce subgroups that are not contained in the isometry group of $S$ for any hyperbolic metric on $S$. 

For many surfaces, this is not simply an artifact of our particular construction.  By choosing the collection  $\{\Omega_v\}$ carefully, we can often ensure that the resulting surface $S$ has a non-displaceable subsurface, and hence its isometry group (with respect to any hyperbolic metric) contains only finite groups \cite[Lemma 4.2]{APV2}.  In particular, the groups we construct could not arise from a construction using isometries for any such surface.

\section{Free groups, wreath products, and Baumslag-Solitar groups}\label{sec:subgroupconstructions}
In this section, we use shift maps and multipushes to construct free groups, certain wreath products, and solvable Baumslag-Solitar groups as subgroups of big mapping class groups.

\subsection{Free groups}\label{sec:free}
The  construction of  Schreier surfaces from \Cref{sec:surfaces}
was motivated by the following construction of a free subgroup of intrinsically infinite type. 

\begin{example}\label{ex:freegroup}
Let $\Gamma$ be the Cayley graph of the free group $\mathbb{F}_2=\<a,b\>$, which is the Schreier graph $\Gamma(\mathbb F_2,\{a,b\},\{\text{id}\})$, and build the Schreier surface $S = S_{\Gamma}(\Pi)$ with $\Pi$ a torus with one boundary component. See \Cref{fig:multipush}. This Schreier surface is homeomorphic to the blooming Cantor tree, that is, the surface with no boundary components, no planar ends, and a Cantor set of nonplanar ends.
The multipushes $x_a$ and $x_b$ generate a copy of $\mathbb{F}_2$ in $\pmap(S)$. To see this, observe that for any $g\in \<a,b\>$, 
the multipush $x_a$ maps 
$\Pi_g$ to $\Pi_{ga}$, 
and similarly for $x_b$. Thus, the only way for a word $w\in\<x_a,x_b\>$ to act trivially on the surface is if the corresponding word in $\<a,b\>$ is trivial. Moreover, \Cref{rem:infinitetype} shows that this copy of $\mathbb{F}_2$ in $\pmapS$ is not contained in $\overline \mapcS$.
\end{example}

In this example, it is straightforward to prove that a non-trivial word $w\in\langle x_a,x_b\rangle$ acts non-trivially on the surface because $\mathbb{F}_2$ has no relations and  
$\Gamma$ is a tree,  so we only need to track where $w$ sends $\Pi_{\text{id}}$. With a more nuanced analysis of the action of $w$, however, we can show that multipushes generate a free group in a much more general setting. 
Recall that the collection $\familya$ of Schreier surfaces was defined in \Cref{def:familya}.

\begin{theorem}\label{thm:freegroup}
	Let $\Gamma$ be a Schreier graph for a triple $(G, T, H)$ and $S$ any associated Schreier surface.  The set $\{\mpush_\alpha \mid \alpha \in T\}$ generates a free group of rank $|T|$ in $\map(S)$. If $\lvert T \rvert=1$ and $\Gamma$ is finite, then we require that at least one $\Omega_v$ is not  a sphere. 
	
	Moreover, when $S \in \familya$, there exist countably many non-conjugate embeddings of such a free group in $\mapS$, none of which can lie entirely in the isometry group for any hyperbolic metric on $S$. If $S$ is not finite type and not the Loch Ness monster surface, these free groups  cannot be completely contained in $\overline \mapcS$.
\end{theorem}

\begin{proof}
Let $w=t_1\ldots t_k$ be a nontrivial, freely reduced word in the free group generated by the set $T$, and let $x_w:=x_{t_k}\cdots x_{t_1}$ be the product of multipushes. We aim to show that $x_w$ is nontrivial in $\mapS$. We first observe that if $\mpush_w(\Pi_{Hg})=\Pi_{Hgw}\neq \Pi_{Hg}$ for any coset $Hg$, then $\mpush_w$ is nontrivial in $\mapS$. We may therefore assume that $\mpush_w$ returns each $\Pi_{Hg}$ to itself.  In particular, this implies that $Hgw=Hg$ for all $g\in G$, and so the edge path given by labels $(t_1, \dots, t_k)$ in $\Gamma(G,T,H)$ based at any vertex describes a cycle.

First consider a one-generated group $G$. If $\Gamma$ is infinite, then we must have $H=\{\text{id}\}$, in which case $\Gamma(G, T, H)=\Gamma(\Z,  \{1\},  \{\operatorname{id}\})$ is the Cayley graph of $\Z$ with its standard generator.  Since this graph has no cycles, each element $x_w$ with $w\in G$ is non-trivial in $\mapS$.  
On the other hand, suppose $\Gamma$ is a finite cycle of order $k$, and consider the multipush $x_t$, where  $t$ is the generator of $G$. Then, $x_t^k$ represents a cycle in $\Gamma$, but the requirement that some $\Omega_i$ is not a sphere guarantees that the curve $\gamma$ and $\mpush_t^k(\gamma)$ cobound a surface with non-trivial topology. See \Cref{fig:non-trivial} for the case $k=3$. Thus $\gamma$ and $\mpush_t^k (\gamma)$ are not homotopic, so $\mpush_t^k$ is non-trivial and $\<x_t\>\cong \Z$. 

\begin{figure} 
    \centering
    \def\svgwidth{3in}
\begingroup%
  \makeatletter%
  \providecommand\color[2][]{%
    \errmessage{(Inkscape) Color is used for the text in Inkscape, but the package 'color.sty' is not loaded}%
    \renewcommand\color[2][]{}%
  }%
  \providecommand\transparent[1]{%
    \errmessage{(Inkscape) Transparency is used (non-zero) for the text in Inkscape, but the package 'transparent.sty' is not loaded}%
    \renewcommand\transparent[1]{}%
  }%
  \providecommand\rotatebox[2]{#2}%
  \newcommand*\fsize{\dimexpr\f@size pt\relax}%
  \newcommand*\lineheight[1]{\fontsize{\fsize}{#1\fsize}\selectfont}%
  \ifx\svgwidth\undefined%
    \setlength{\unitlength}{196.79280113bp}%
    \ifx\svgscale\undefined%
      \relax%
    \else%
      \setlength{\unitlength}{\unitlength * \real{\svgscale}}%
    \fi%
  \else%
    \setlength{\unitlength}{\svgwidth}%
  \fi%
  \global\let\svgwidth\undefined%
  \global\let\svgscale\undefined%
  \makeatother%
  \begin{picture}(1,0.89207392)%
    \lineheight{1}%
    \setlength\tabcolsep{0pt}%
    \put(0.46499545,0.78969876){\makebox(0,0)[t]{\lineheight{1.25}\smash{\begin{tabular}[t]{c}$\Pi$\end{tabular}}}}%
    \put(0,0){\includegraphics[width=\unitlength,page=1]{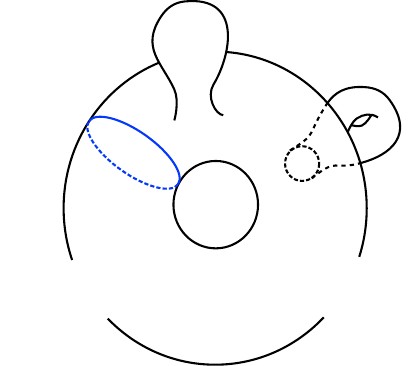}}%
    \put(0.18654448,0.14800438){\makebox(0,0)[t]{\lineheight{1.25}\smash{\begin{tabular}[t]{c}$\Pi$\end{tabular}}}}%
    \put(0.86250734,0.14800438){\makebox(0,0)[t]{\lineheight{1.25}\smash{\begin{tabular}[t]{c}$\Pi$\end{tabular}}}}%
    \put(0,0){\includegraphics[width=\unitlength,page=2]{non-trivial.pdf}}%
    \put(0.17029403,0.61647685){\makebox(0,0)[lt]{\lineheight{1.25}\smash{\begin{tabular}[t]{l}{\color{blue}$\gamma$}\end{tabular}}}}%
    \put(-0.00515096,0.36771546){\makebox(0,0)[lt]{\lineheight{1.25}\smash{\begin{tabular}[t]{l}{\color{red}$x_t^3(\gamma)$}\end{tabular}}}}%
  \end{picture}%
\endgroup%

    \caption{A surface $S$ built from the Cayley graph of $\Z/3\Z=\langle t\rangle$. The curve $\gamma$ is not homotopic to  its image under $x_t^3$  due to the handle on the back of $S$.}
    \label{fig:non-trivial}
\end{figure}

Now assume $|T|=n\geq 2$, so that every vertex of $\Gamma(G, T, H)$ has degree $2n\geq 4$. Let $p\colon \tilde{\Gamma}\arr \G$ be the universal cover of the labelled graph $\Gamma$, which is a tree of valency $2n$ with edge labels in the set $T$. Construct the Schreier surface $\tilde S=S_{\tilde{\Gamma}}(\Pi) \,\underset{\tilde{v}\in V(\tilde{\G})}{\#}\, \Omega_{\tilde{v}}$, where $\Omega_{\tilde{v}}=\Omega_{v}$ whenever $p(\tilde{v})=v$. By construction, $\tilde S$ is a cover of $S=S_{\Gamma}(\Pi)\connsumv \Omega_v$. See \Cref{fig:CoveringSurface} for an example.

For each $t\in T$, let $\tilde{x}_t$ be the multipush on $\tilde S$ 
obtained by identifying $\tilde{\Gamma}$ with the Cayley graph of the free group with basis $T$. The covering map $P\colon \tilde S\rightarrow S$ 
induces a homomorphism from the group generated by the multipushes on $\tilde S$ 
 to the group generated by the multipushes on $S$ 
 by mapping $\tilde{x}_t\mapsto x_t$. Recall that, by assumption, $w=t_1\ldots t_k$ is a non-trivial reduced word in the free generating set $T$ and  $x_w=x_{t_k}\dots x_{t_1}$. Let $\tilde{x}_w=\tilde{x}_{t_k}\dots \tilde{x}_{t_1}$. 

Suppose towards a contradiction that $x_w$ is trivial in $\mapS$.  Then the following commutative diagram of homeomorphisms shows that $\tilde{x}_w$ is a deck transformation.

\[
\begin{tikzcd}
S_{\tilde{\Gamma}}(\Pi) \arrow{r}{\tilde{x}_w} \arrow[swap]{d}{P} & S_{\tilde{\Gamma}}(\Pi) \arrow{d}{P} \\
S_{\Gamma}(\Pi) \arrow{r}{x_w = \text{id}} & S_{\Gamma}(\Pi)
\end{tikzcd}
\]

\begin{figure}[h] 
    \centering
    \def\svgwidth{8in}
    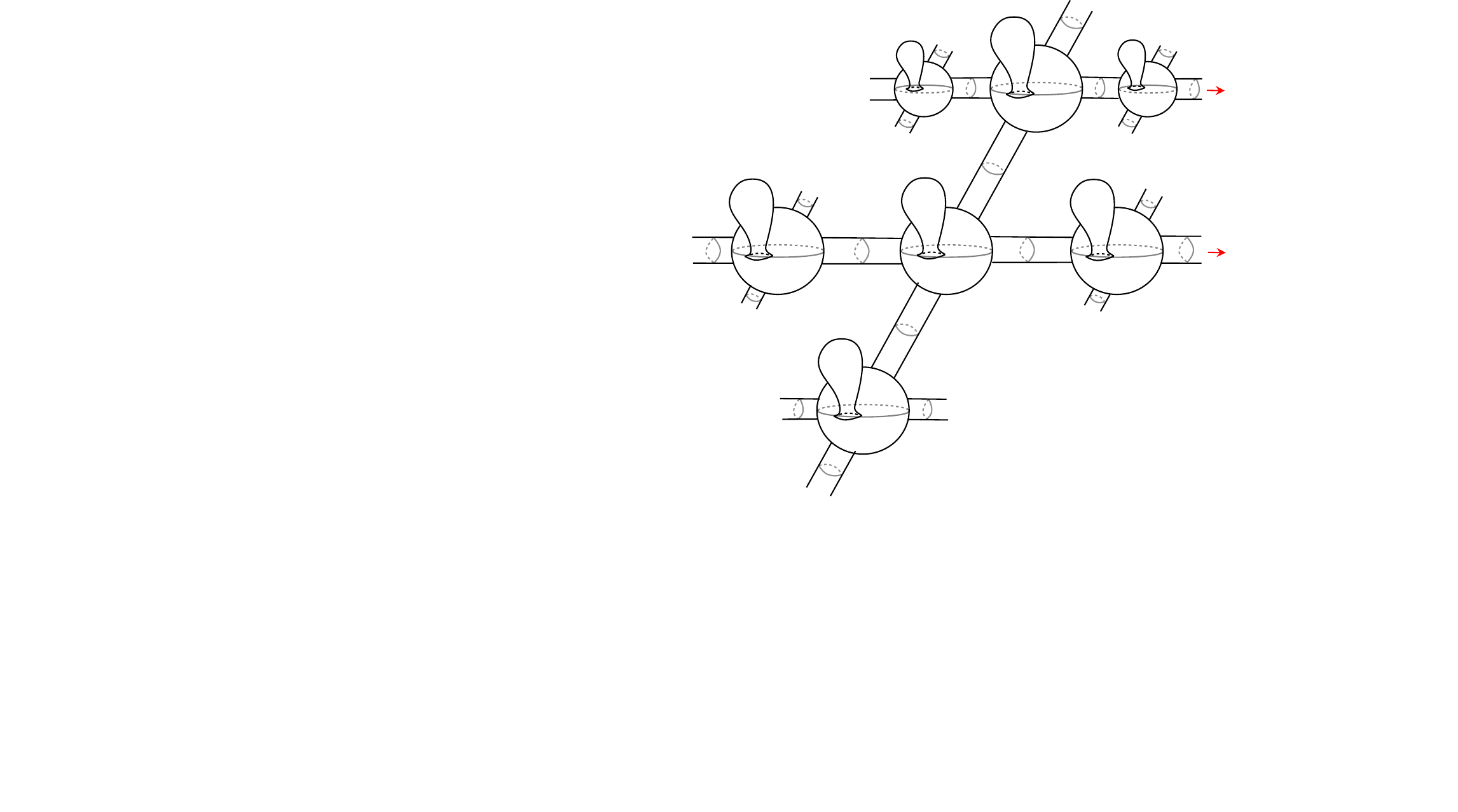
    \caption{An example of the surface $\tilde S=S_{\tilde\Gamma}(\Pi)$ and lifts of multipushes $x_a,x_b$. }
    \label{fig:CoveringSurface}
\end{figure}

On the other hand, since $\tilde{x}_w$ is a multipush, it moves  every vertex surface of $\tilde S$ 
at most $k$ steps away from itself, a bounded distance. We claim this is a contradiction. Indeed, as the covering map sends vertex surfaces to vertex surfaces and edge surfaces to edge surfaces, respecting the edge labels in $T$, we see that any deck transformation of $P\colon \tilde S \arr S$ 
is determined by a deck transformation of the covering $p\colon  \tilde{\Gamma}\rightarrow \Gamma$. One readily checks that for any such nontrivial deck transformation and for all $j\geq 1$, there exists a vertex $v$ in the tree $\tilde\Gamma$ such that the distance from $v$ to its image is larger than $j$, and we have obtained our contradiction.

When $S\in \familya$, it follows from \Cref{lem:nonconjugate} that there are countably many non-conjugate embeddings of the free group $\mathbb F_{|T|}$ in $\mapS$.  By the argument in \Cref{sec:non-isometric}, none of these embeddings lie in the isometry group for any hyperbolic metric on $S$.  Finally, when $S$ is not finite-type or the Loch Ness Monster (in which case $\overline{\mapcS}=\mapS$), each multipush in the argument above is a collection of shift maps, so  \Cref{rem:infinitetype} completes the proof. \end{proof}

    It follows from the proof of \Cref{thm:freegroup} that the support of every non-trivial element of $\mathbb F_{|T|}$ is not contained in the union of the vertex surfaces.  This is clear if $w$ does not fix every $\Pi_{Hg}$, because the shift domains are contained in the support of $x_w$.  On the other hand, suppose $w$ fixes each $\Pi_{Hg}$.  Since $x_w$ is a collection of pushes, it  therefore restricts to the identity on each vertex surface.  However, the proof of the theorem shows that $x_w$ is a non-trivial homeomorphism, and so the support of $x_w$ cannot be contained in the union of the vertex surfaces. See \Cref{fig:loopimage} for an example of what the image of a loop $\gamma$ might look like after the application of $x_w$ when $w$ is trivial in $G$.  This will be a crucial ingredient in the proof of \Cref{cor:otherS}.

     \begin{figure}
     \centering\def\svgwidth{5in}
\begingroup%
  \makeatletter%
  \providecommand\color[2][]{%
    \errmessage{(Inkscape) Color is used for the text in Inkscape, but the package 'color.sty' is not loaded}%
    \renewcommand\color[2][]{}%
  }%
  \providecommand\transparent[1]{%
    \errmessage{(Inkscape) Transparency is used (non-zero) for the text in Inkscape, but the package 'transparent.sty' is not loaded}%
    \renewcommand\transparent[1]{}%
  }%
  \providecommand\rotatebox[2]{#2}%
  \newcommand*\fsize{\dimexpr\f@size pt\relax}%
  \newcommand*\lineheight[1]{\fontsize{\fsize}{#1\fsize}\selectfont}%
  \ifx\svgwidth\undefined%
    \setlength{\unitlength}{642.84468894bp}%
    \ifx\svgscale\undefined%
      \relax%
    \else%
      \setlength{\unitlength}{\unitlength * \real{\svgscale}}%
    \fi%
  \else%
    \setlength{\unitlength}{\svgwidth}%
  \fi%
  \global\let\svgwidth\undefined%
  \global\let\svgscale\undefined%
  \makeatother%
  \begin{picture}(1,0.47242891)%
    \lineheight{1}%
    \setlength\tabcolsep{0pt}%
    \put(0,0){\includegraphics[width=\unitlength,page=1]{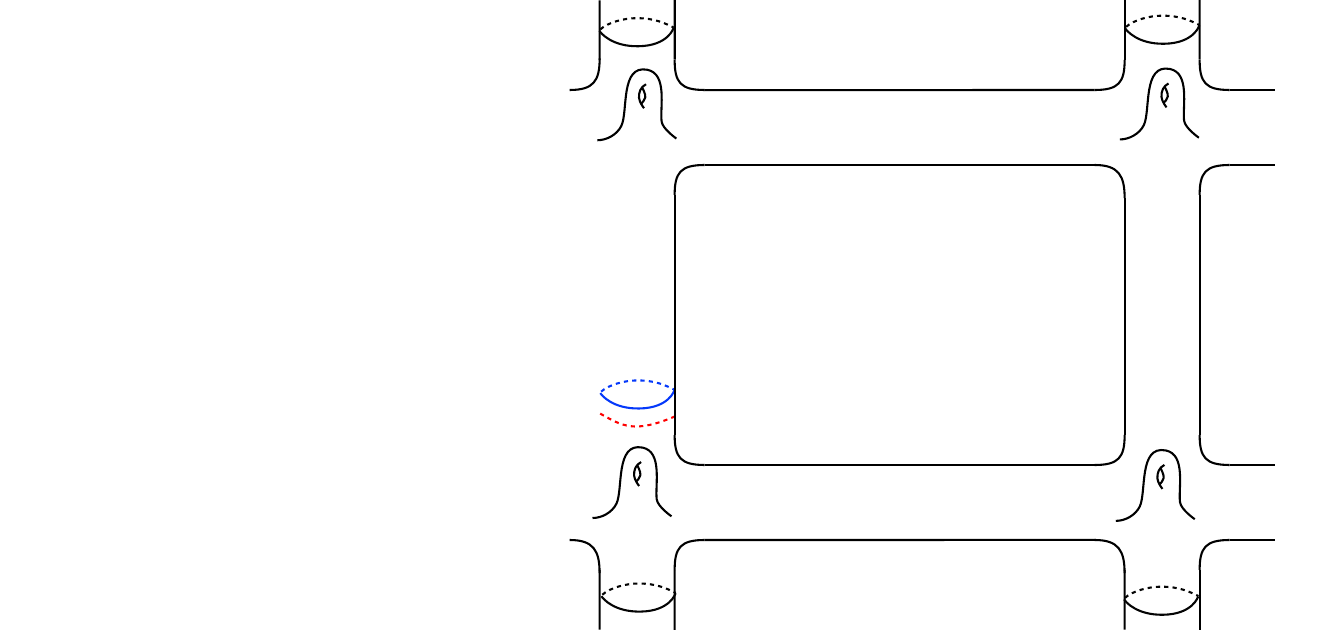}}%
    \put(0.51661554,0.17525578){\makebox(0,0)[lt]{\lineheight{1.25}\smash{\begin{tabular}[t]{l}{\color{blue}$\gamma$}\end{tabular}}}}%
    \put(0.75415214,0.37669711){\makebox(0,0)[lt]{\lineheight{1.25}\smash{\begin{tabular}[t]{l}{\color{red} $x_w(\gamma)$}\end{tabular}}}}%
    \put(0,0){\includegraphics[width=\unitlength,page=2]{schreir.pdf}}%
  \end{picture}%
\endgroup%

     \caption{A portion of the Schreier surface for $(\Z^2,\{a,b\},\{1\})$ and the image of the curve $\gamma$ under the element $x_{bab^{-1}a^{-1}}$.}
    \label{fig:loopimage}
\end{figure}

\subsection{Shift Maps that do not generate a free group}\label{sec:notfree}
The construction above uses a countable collection of intersecting push maps to ensure the resulting group is free. The following example demonstrates why this is necessary by showing that the group generated by two shift maps with minimal intersection is not free. We use the convention that $[x,y]=xyx\inv y\inv$ and choose a right action.

\begin{figure}
    \centering
    \def\svgwidth{2.5in}
\begingroup%
  \makeatletter%
  \providecommand\color[2][]{%
    \errmessage{(Inkscape) Color is used for the text in Inkscape, but the package 'color.sty' is not loaded}%
    \renewcommand\color[2][]{}%
  }%
  \providecommand\transparent[1]{%
    \errmessage{(Inkscape) Transparency is used (non-zero) for the text in Inkscape, but the package 'transparent.sty' is not loaded}%
    \renewcommand\transparent[1]{}%
  }%
  \providecommand\rotatebox[2]{#2}%
  \newcommand*\fsize{\dimexpr\f@size pt\relax}%
  \newcommand*\lineheight[1]{\fontsize{\fsize}{#1\fsize}\selectfont}%
  \ifx\svgwidth\undefined%
    \setlength{\unitlength}{325.88591343bp}%
    \ifx\svgscale\undefined%
      \relax%
    \else%
      \setlength{\unitlength}{\unitlength * \real{\svgscale}}%
    \fi%
  \else%
    \setlength{\unitlength}{\svgwidth}%
  \fi%
  \global\let\svgwidth\undefined%
  \global\let\svgscale\undefined%
  \makeatother%
  \begin{picture}(1,0.9279277)%
    \lineheight{1}%
    \setlength\tabcolsep{0pt}%
    \put(0,0){\includegraphics[width=\unitlength,page=1]{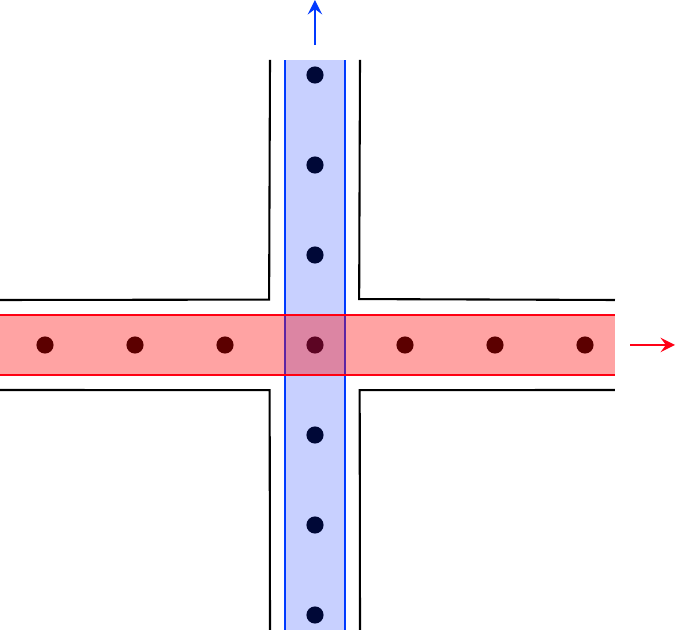}}%
    \put(0.93776879,0.36797148){\makebox(0,0)[lt]{\lineheight{1.25}\smash{\begin{tabular}[t]{l}{\color{red}$h_a$}\end{tabular}}}}%
    \put(0.48653599,0.88456994){\color[rgb]{0.91372549,0.68627451,0.68627451}\makebox(0,0)[lt]{\lineheight{1.25}\smash{\begin{tabular}[t]{l}{\color{blue}$h_b$}\end{tabular}}}}%
    \put(0.41695337,0.36375437){\makebox(0,0)[lt]{\lineheight{1.25}\smash{\begin{tabular}[t]{l}\tiny$(0,0)$\end{tabular}}}}%
  \end{picture}%
\endgroup%

    \caption{The shifts $h_a$ and $h_b$ do not generate a free group.}
    \label{fig:notfree}
\end{figure}

Let $\Gamma$ be the four-ended tree with a single vertex of valence four and all other vertices of valence two. Identify $\G$ with the coordinate axes in $\R^2$ to get a labeling of the vertices as integer coordinates. Let $\Pi$ be any surface with one boundary component that is not a disk,  and construct the surface $S=S_{\G}(\Pi)$. There is a horizontal shift $h_{a}$ corresponding to the $+(1,0)$ map on the $x$--axis and a vertical shift $h_{b}$ corresponding to the $+(0,1)$ map on the $y$--axis, as shown in \Cref{fig:notfree}. The intersection of the supports of these shifts is contained in the front of $V_{(0,0)}$, the vertex surface at $(0,0)$. It can be checked that the support of $[h_a,h_b]$ is contained in the fronts of $V_{(-1,0)}$, $V_{(0,0)}$, and $V_{(0,-1)}$ and the adjoining edge surfaces. The word $w=h_ah_bh_a^2$ maps $\{\Pi_{(0,-1)}, \Pi_{(-1,0)}, \Pi_{(0,0)}\}$ to the collection $\{ \Pi_{(1,0)}, \Pi_{(2,0)}, \Pi_{(3,0)} \}$. Thus, the elements $[h_a,h_b]$ and $w[h_a,h_b]w\inv$ have disjoint supports and so commute.  More generally, the words $w_n= h_a^{3n+1}h_bh_a^{2}$ map $\{\Pi_{(0,-1)}, \Pi_{(-1,0)}, \Pi_{(0,0)}\}$ to $\{ \Pi_{(1+3n,0)}, \Pi_{(2+3n,0)}, \Pi_{(3+3n,0)}\}$. From this, we see that $H:=\<h_a,h_b\>$ is not a free group and actually contains copies of $\Z^n$ for all $n$.  

In fact, $H$ is isomorphic to a 2--generated subgroup of an infinite strand braid group.  To see this, note that the group structure of $H$ is not dependent on the surface $\Pi$ that we attach, so we may assume $\Pi$ is a punctured disk. We can also realize each shift domain as a disk with countably many punctures with two distinct accumulation points on the boundary. Because braid groups are mapping class groups of punctured disks, this viewpoint allows us to realize $H$ as a subgroup of the infinite strand braid group in which braids are allowed to have non-compact support. In particular, $H$ is isomorphic to the subgroup of this braid group generated by the elements $h_a$ and $h_b$, viewed as braids with non-compact support.

\subsection{Wreath products} \label{sec:wreath}

Recall that if $H$ acts on a set $\Lambda$, then the (restricted) wreath product $G\wr_{\Lambda} H$ is defined as
\[
G\wr_{\Lambda} H= G^{\Lambda} \rtimes_\gamma H,
\] 
that is, the semidirect product of $H$ with the direct sum of copies of $G$ indexed by $\Lambda$. Here, $G^\Lambda=\oplus_{\Lambda} G$ and is the set of $(g_\lambda)_{\lambda\in\Lambda}$. The automorphism $\gamma\colon H\to \Aut(G^\Lambda)$ is defined by $\gamma(h)(G_\lambda)=hG_\lambda h^{-1}=G_{h\lambda}$, so that $H$ acts on $G^\Lambda$ by permuting the coordinates according to the action on the indices. When it is clear from context, or when $\Lambda=H$, we may simply write $G\wr H$.  

We now construct a collection of wreath products in big mapping class groups. The most straightforward example of this construction is when $S$ is a surface which admits a shift  whose domain is an embedded copy of $D_\Pi$ for some surface $\Pi$ with one boundary component. For any $G\leq \map(\Pi)$, we generalize a construction of Lanier and Loving \cite{LanierLoving} to construct $G\wr \Z$ as a subgroup of $\map(S)$. When $G$ is chosen to be the infinite cyclic group generated by a single Dehn twist,  we recover \cite[Theorem 4]{LanierLoving}.

\begin{proposition}\label{prop:GwrHgeneral}
    Let $G\leq \map(\Pi)$, where $\Pi$ is a surface with a single boundary component.  Let $S$ be a surface and $H\leq \mapS$ be generated by a finite collection of pushes and multipushes, all of whose domains are (unions of) embedded copies of $A_\Pi$ or $D_{\Pi}$. Index the copies of $\Pi$  in these domains by $\Lambda$.  The wreath product $G\wr_{\Lambda} H$ is a subgroup of $\mapS$.
\end{proposition}

\begin{proof}
    Let $h_1,\dots, h_n$ be the generators of $H$, so $\Lambda$ is a set indexing the copies of $\Pi$ contained in the union of the domains of the $h_i$.  Each $h_i$ permutes the copies of $\Pi$ in its domain and so acts on $\Lambda$: if $\lambda\in \Lambda$, then $h_i(\lambda)$ is defined to be the index of $h_i\left(\Pi_\lambda\right)$.  This induces an action of $H$ on $\Lambda$. 
    
     Let $G\leq \map(\Pi)$, and let $G_\lambda\cong G$ be the corresponding subgroup of $\mapS$ supported on $\Pi_\lambda$.  Whenever $\lambda\neq\lambda'$, the subgroups $G_\lambda$ and $G_{\lambda'}$ have disjoint supports and commute, so $\<G_{\lambda}\mid \lambda \in \Lambda \>=G^{\Lambda}$.  For any $h\in H$ and $\lambda\in \Lambda$, we have
    $ hG_{\lambda} h^{-1} = G_{h(\lambda)}$  and $ H\cap  G_{\lambda}=\{1\}$.  Therefore, the subgroup of $\mapS$ generated by $\langle H, G_{\lambda}\mid \lambda \in \Lambda\rangle$ is isomorphic to $G\wr_{\Lambda} H$.
\end{proof}

We illustrate this proposition with several examples.

\begin{example} \label{ex:GwrHgeneral} 
\Cref{prop:GwrHgeneral} applies whenever $S$ and $H$ are one of the following.
    \begin{enumerate}
        \item Let $S$ be a surface with an embedded copy of $D_\Pi$, and let $H$ be generated by a (possibly one-ended) shift $h$, so that $H\cong \Z$.    The index set $\Lambda$ is simply $\mathbb Z$, and $h$ acts on $\Lambda$ as addition by 1.
        \item Let $S$ be a Schreier surface for a triple $(A,T,B)$ such that $t_1,\dots, t_n\in T$ correspond to biinfinite geodesics in $\Gamma(A,T,B)$. Let $H$ be the subgroup of $\mapS$ generated by the multipushes $x_{t_1},\dots, x_{t_n}$.  By \Cref{thm:freegroup}, $H\cong \mathbb F_n$.  In this case, the index set $\Lambda$ is the collection of right cosets $\{Ba\mid a\in A\}$.  Each generator $x_{t_i}$ acts on $\Lambda$ as follows: if $Ba\in \Lambda$, then $x_{t_i}\cdot Ba = Bat_i$.
        \item Let $S=S_\Gamma(\Pi)$ be the surface described in \Cref{sec:notfree}, and let $H=\langle h_a,h_b\rangle$ be the subgroup of $\map(S)$ constructed in that section.  In this case, $H$ is not free.  The index set $\Lambda$ is the set $\{(0,n),(n,0)\mid n\in\mathbb Z\}$, and the generators $h_a$ and $h_b$ act on $\Lambda$ as addition by $(1,0)$ and $(0,1)$, respectively. 
    \end{enumerate}

\end{example}

When $S\in\familya$, it follows from \Cref{lem:nonconjugate} that there are countably many non-conjugate embeddings of $G \wr H$ in $\mapS$ for $G,H$ as in the statements of \Cref{prop:GwrHgeneral}.  Moreover, none of these embeddings lie in the isometry group for any hyperbolic metric on $S$  by the discussion in \Cref{sec:non-isometric} or in  $\overline{\mapcS}$ by \Cref{rem:infinitetype}.

\subsection{Solvable Baumslag-Solitar groups} \label{sec:BS1n}

For our third and final construction in this section, we focus on solvable Baumslag-Solitar groups. Fixing a positive integer $n$, recall that the Baumslag-Solitar group $BS(1,n)$ is the group with presentation 
\[
BS(1,n)=\langle a,t \mid tat^{-1}a^{-n} \rangle.
\]

Given $n\in \mathbb N$, we first construct a shift map with $n^k$th roots for all $k\in \mathbb N$ on a certain class of surfaces.  To do this, we borrow ideas from asymptotically rigid mapping class groups \cite{AF21}.  We say that a mapping class group element is \emph{rigid} with respect to some pants decomposition if the pants curves, seams, and fronts (as defined in \Cref{defn:seams}) are preserved setwise by the element.

\begin{lemma}\label{lem:roots}
Let $S$ be a surface whose end space contains a clopen subset homeomorphic to a Cantor set of planar ends. For any $n\in \N$, there is a shift map in $\mapS$ with $n^k$th roots for all $k\in \N$.
\end{lemma}

\begin{proof}

If the end space of $S$ has a clopen subset homeomorphic to a Cantor set of planar ends, then there exists a subsurface $S'\subset S$ which is homeomorphic to a Cantor tree surface with a boundary component (see \cite[Lemma 2.7]{AMP21} for a proof). That is, $S'$ has exactly one compact boundary component, genus zero, and $E(S')$ is a Cantor set.  We will construct an explicit realization of $S'$ and a shift map that has roots. Then, because homeomorphisms induce isomorphisms of the respective mapping class groups, we see that any surface with a subsurface homeomorphic to $S'$ also has a shift map with the desired property, proving the theorem in its full generality.

We start with a pair of pants with countably many cuffs, called $B$, which is constructed from countably many copies of a standard pair of pants (3-hold spheres) in the following way. Let $P$ be a standard pair of pants with boundary components labeled by $\alpha$, $\beta_r$, $\beta_\ell$ and equipped with seams so that $P$ has a back side and a front side. Then $\beta_r$ and $\beta_\ell$ are the right and left cuff of $P$, respectively. The surface $B$ is constructed from countably many copies of $P$, labeled by $\{P_i\}_{i\in\Z}$, with the right cuff of $P_i$ glued to the left cuff of $P_{i+1}$ for all $i \in \Z$, such that the seams and fronts of the $P_i$ are compatible. Here compatibility is the property described after \Cref{defn:seams}, i.e.,  the union of the seams separates $B$ into two disjoint connected components, the front and the back, containing the front and, respectively, the back of each $P_i$. The boundary component of $B$ corresponding to the cuff $\alpha$ of $P_i$ will be labeled by $\alpha_i$.

Next, let $T$ be a Cantor tree surface with boundary that is built from $(n+1)$-holed spheres, called pants, equipped with compatible seams and fronts.
Glue a copy of $T$, called $T_0$, to $\alpha_0$ in $B$ so that the seams and front line up with those of $B$. We will use finite strings in $\{1,\dots, n\}$ to label the curves in the pants decomposition of $T_0$ as follows. Start with the pants containing the boundary component of $T_0$ and, moving clockwise along the boundary of the polygon that forms the front of the pants, label the other boundary curves $\alpha_{0,1}, \dots \alpha_{0,n}$. Continue on adjacent pants, using the cyclic order to append the corresponding digit in $\{1,\dots, n\}$. See \Cref{fig:PantsDecomp} below for the case $n=3$.
Finally, glue a copy of $T$, called $T_i$, to each other curve $\alpha_i$ of $B$, so that the seams and fronts line up. Label the pants curves of $T_i$ by $\alpha_{i,x}$ with $x\in \{1,\dots,n\}^{<\N}$ analogous to the labeling of pants curves of $T_0$.

\begin{figure}[h]
        \centering
        \begin{overpic}[width=0.6\linewidth]{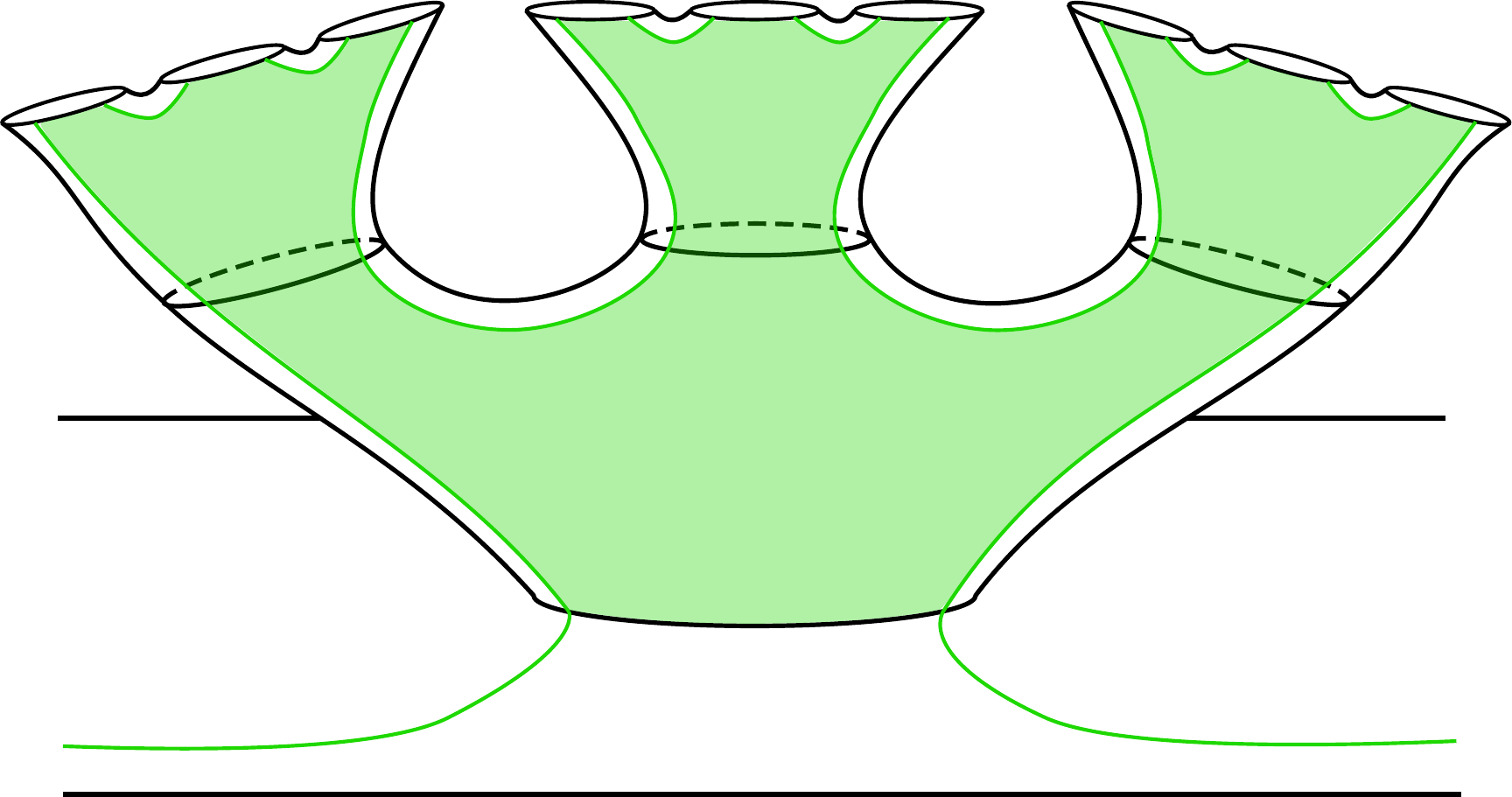}
            \put(49,8){$\alpha_0$}
            \put(18,31){$\alpha_{0,1}$}
            \put(49,32){$\alpha_{0,2}$}
            \put(77,31){$\alpha_{0,3}$ }
            \put(1,48){\tiny{$\alpha_{0,11}$}}
            \put(11,50.5){\tiny{$\alpha_{0,12}$}}
            \put(22,53){\tiny{$\alpha_{0,13}$}}
            \put(36,54){\tiny{$\alpha_{0,21}$}}
            \put(46,54){\tiny{$\alpha_{0,22}$}}
            \put(58,54){\tiny{$\alpha_{0,23}$}}
            \put(72,53){\tiny{$\alpha_{0,31}$}}
            \put(83,50.5){\tiny{$\alpha_{0,32}$}}
            \put(93,47.5){\tiny{$\alpha_{0,33}$}}
            \put(46,20){front}
        \end{overpic}
        \caption{The labeling of curves in the pants decomposition for $n=3$ on $T_0$.}
        \label{fig:PantsDecomp}
\end{figure} 

The resulting surface $R$ is homeomorphic to a Cantor tree surface since it has no boundary, genus zero, and its endspace is homeomorphic to a Cantor set of planar ends.

There is a rigid shift $\phi\in\map(R)$ with shift domain a properly embedded copy of $D_T$, which shifts $T_i$ to $T_{i+1}$, and which respects the seams and front of $R$.
By construction, $\phi$ has an $n$th root which shifts $\alpha_{i,j}$ to $\alpha_{i,j+1}$, when $1\leq j<n$ and $\alpha_{i,n}$ to $\alpha_{i+1,1}$. See \Cref{fig:shiftroots} for an example in the case where $n=3$.
    \begin{figure}[h]
        \centering
        \begin{overpic}[width=0.8\linewidth]{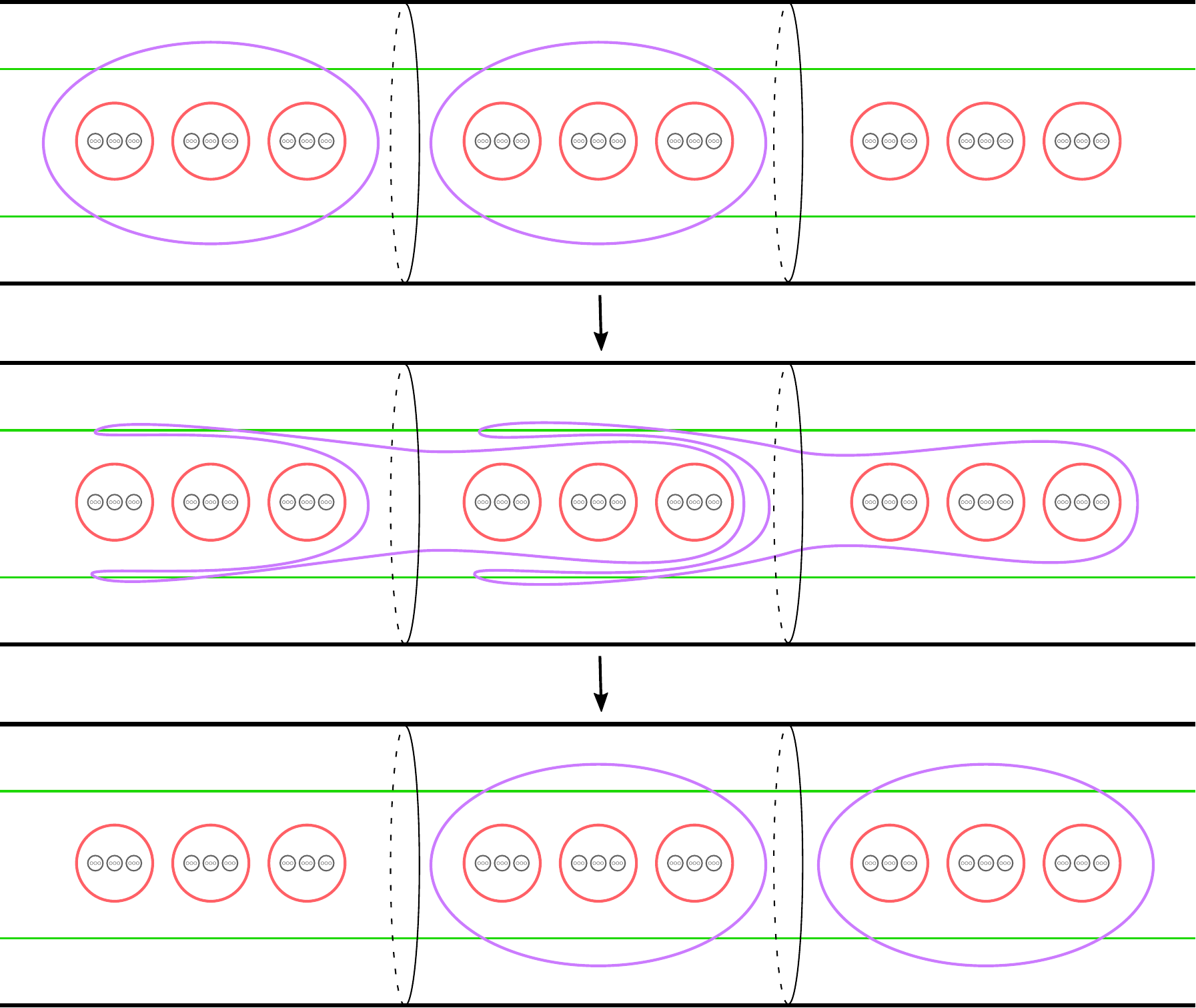}
        \put(17.5, 62){\textcolor{violet}{$\alpha_0$}}
           \put(50, 62){\textcolor{violet}{$\alpha_1$}}
              \put(20, 33){\textcolor{violet}{$f^3(\alpha_0)$}}
                 \put(55, 33){\textcolor{violet}{$f^3(\alpha_1)$}}
                \put(7.5,76.5){\textcolor{purple}{\tiny{$\alpha_{0,1}$}}}
                 \put(16,76.5){\textcolor{purple}{\tiny{$\alpha_{0,2}$}}}
                \put(23,76.5){\textcolor{purple}{\tiny{$\alpha_{0,3}$}}}
                   \put(41,76.5){\textcolor{purple}{\tiny{$\alpha_{1,1}$}}}
                   \put(35.5,45){\textcolor{purple}{\scalebox{.4}{$f^3(\alpha_{0,1})$}}}
                    \put(47,46){\textcolor{purple}{\scalebox{.4}{$f^3(\alpha_{0,2})$}}}
                    \put(54,46){\textcolor{purple}{\scalebox{.4}{$f^3(\alpha_{0,3})$}}}
                 \put(49,76.5){\textcolor{purple}{\tiny{$\alpha_{1,2}$}}}
                \put(57,76.5){\textcolor{purple}{\tiny{$\alpha_{1,3}$}}}
                 \put(52,56){$f^3$}
                 \put(52,26){isotopy}
                       \put(41, 1.25){\footnotesize{\textcolor{violet}{$f^3(\alpha_0) = \phi(\alpha_0)$}}}
                 \put(76, 1.25){\footnotesize{\textcolor{violet}{$f^3(\alpha_1) = \phi(\alpha_1)$}}}
                 \put(12,68){\textcolor{purple}{$\rcurvearrowright$}}
                 \put(12.5, 66){\textcolor{purple}{\tiny $f$}}
                  \put(21,68){\textcolor{purple}{$\rcurvearrowright$}}
                 \put(21.5, 66){\textcolor{purple}{\tiny $f$}}
                 \put(31.5,67){\textcolor{purple}{\LARGE{$\rcurvearrowright$}}}
                 \put(32.5, 65){\textcolor{purple}{\scriptsize $f$}}
                 \put(45,68){\textcolor{purple}{$\rcurvearrowright$}}
                 \put(45.5, 66){\textcolor{purple}{\tiny $f$}}
                  \put(53,68){\textcolor{purple}{$\rcurvearrowright$}}
                 \put(53.5, 66){\textcolor{purple}{\tiny $f$}}
        \end{overpic}
        \caption{A flattened picture of the surface $R$ for $n=3$. Let $f$ be the map that shifts $\alpha_{i,j}$ to $\alpha_{i,j+1}$, when $1\leq j<3$ and $\alpha_{i,3}$ to $\alpha_{i+1,1}$. The domain of $f$ is shown in green. We see that $f^3$ is indeed equal to the rigid shift $\phi$ in $\map(R)$ that takes the $T_i$ to $T_{i+1}$.}
        \label{fig:shiftroots}
\end{figure} 

Similarly, the $n^k$th root shifts the curves at height $k$ according to the lexicographic ordering on $\N\times \{1,\dots,n\}^{k}$. Note that $\phi$ being rigid is part of what ensures these roots of $\phi$ exist. Finally, note that the domain of $\phi$ is not all of $R$ so that we may remove an open disk, say on the back of $R$, to obtain a surface homeomorphic to $S'$. The corresponding mapping class of $S'$, which we also call $\phi$ by an abuse of notation, is a shift map with $n^k$th roots for all $k \in \N$.
\end{proof}

\begin{theorem}
    \label{thm:BS1n}
 Let $S$ be a surface whose end space contains a clopen subset homeomorphic to
a Cantor set of planar ends. Then $BS(1,n)\leq \mapS$ for all $n >0$.
\end{theorem}

\begin{proof}
    Let $S'\subset S$ be homeomorphic to the Cantor tree surface with a boundary component.
    Let $\push\in \map(S')$ be a rigid shift map with domain $D_{T}$, where $T$ is also homeomorphic to a Cantor tree surface with a boundary component. Index the copies of $T$ in $D_{T}$ by $\Z$ according to the shift order. 
    
    Now, since $T_0$ is homeomorphic to a Cantor tree surface with one boundary component, we may choose a rigid shift $\phi$ on $T_0$ with $n^k$th roots by \Cref{lem:roots}. For $k\geq0$ let $\phi_k=\sqrt[n^k]{\phi}$, and for $k<0$ let $\phi_k=\phi^{n^{-k}}$. Note that for all $k \in \Z$, we have $\phi_k^n=\phi_{k-1}$.

    Next, let $\tilde{\phi_k}=h^k\phi_k h^{-k}$. Heuristically, $\tilde{\phi_k}$ performs $\phi_k$ on $T_k$ instead of on $T_0$. Because the domains of $\tilde{\phi_k}$ are disjoint, the product \[\Phi=\prod_{k\in Z}\tilde{\phi_k}\] is a well-defined element of $\mapS$. This map simultaneously performs $\phi_k$ on $T_k$ for all $k \in \Z$. See \Cref{fig:ModifiedPantsDecomp} for an example of the action of $\Phi$ on $S$ when $n=2$. 

    \begin{figure}[h]
        \centering
        \begin{overpic}[width=0.9\linewidth]{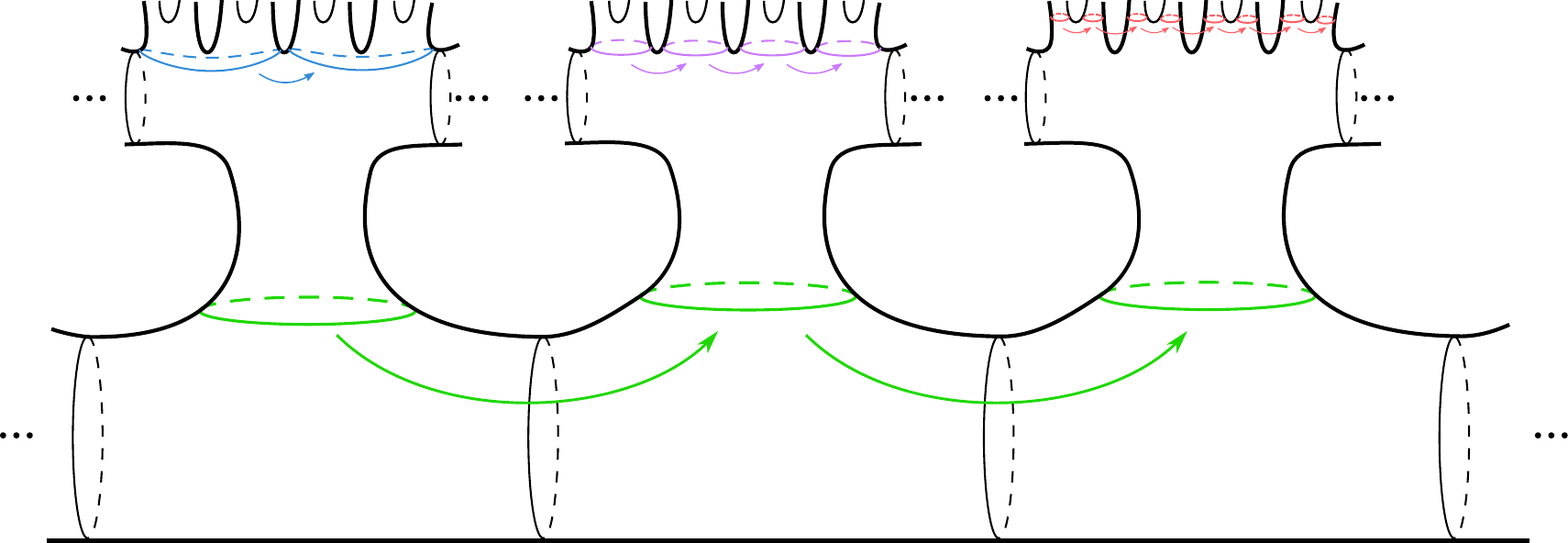}
        \put(17.5,26.5){\textcolor{cyan}{$\phi^2$}}
          \put(46,27.5){\textcolor{violet}{$\phi$}}
          \put(72,29.5){\textcolor{purple}{\footnotesize{$\phi^{\frac{1}{2}}$}}}
            \put(31,6){\textcolor{green}{$h$}}
            \put(60,6){\textcolor{green}{$h$}}
               \put(18,18){$T_{-1}$}
             \put(47,18){$T_{0}$}
              \put(77,18){$T_{1}$}
        \end{overpic}
        \caption{This figure shows the action of the shift $h$ on $S$ as well as the action of $\Phi$ on the subsurfaces $T_{-1}, T_0, T_1$ when $n=2$.}
        \label{fig:ModifiedPantsDecomp}
\end{figure}

    Let  $f\colon BS(1,n)\to \map(S)$ be the map defined by $f(a)=\Phi$ and $f(t)=\push$.  We will show that  $f$ is an isomorphism onto its image, i.e., $\mapS$ contains a subgroup isomorphic to $BS(1,n)$. 

    For each $k\in \Z$, the mapping class $f(tat^{-1})=\push \Phi \push^{-1}$ first shifts $T_k$ to $T_{k-1}$, applies $\Phi$, which now acts as $\phi_{k-1}=\phi_k^n$ on $T_{k-1}$, and then shifts $T_{k-1}$ back to $T_k$. It follows that 
    \[
        f(tat^{-1})=\push \Phi \push^{-1}=\Phi^n=f(a^n).
    \]  
    Therefore, $f$ is a well-defined homomorphism.  This can also be verified algebraically:

\begin{align*}
    \Phi^n&=\prod_{k\in Z}\tilde{\phi_k}^n=\prod_{k\in Z}(h^k\phi_k h^{-k})^n=\prod_{k\in Z}h^k\phi_k^n h^{-k} = \prod_{k\in Z}h^k\phi_{k-1} h^{-k}\\
    &=h \left( \prod_{k\in Z}h^{k-1}\phi_{k-1} h^{-(k-1)}\right) h\inv= h \Phi h\inv.
\end{align*}

    Suppose there exists $g\in BS(1,n)$ such that $f(g)$ is the identity of $\map(S)$.   Using the relation in $BS(1,n)$, the element $g$ can be written as $g=t^ia^kt^{-j}$ for some $k\in \Z$ and $i,j\in \Z_{\geq 0}$.  Since $f(g)=\push^i\Phi^k\push^{-j}$ is the identity, it must fix each $T_i$, and so we must have $i=j$.  Consider the surface $T_0$.  Then, $f(g)$ first shifts $T_0$ to the left $j$ times, applies $\Phi^k$ (which acts as $\phi_{-j}^k$), and then shifts back to $T_0$.  The result is that $f(g)$ acts by $\phi_{-j}^k$ on $T_0$. The only way that $f(g)$ can act as the identity on $T_0$ is if $k=0$.  Thus, $g=t^ja^0t^{-j}=1$, and $f$ is injective, as desired.
\end{proof}

The construction above embeds solvable Baumslag-Solitar groups into mapping class groups of certain 
\emph{infinite-type} surfaces.  This is in contrast to the finite-type case, where $BS(1,n)$ is never a subgroup of the mapping class group due to the Tits alternative: every subgroup of such a mapping class group either contains a free subgroup or is virtually abelian  \cite{Ivanov, McCarthy}.  Since $BS(1,n)$ is solvable, it does not contain any free subgroups, but it is also not virtually abelian.

We expect that the techniques in \Cref{sec:non-conjugate} can be used to produce countably many non-conjugate embeddings of $BS(1,n)$ into $\mapS$ if the topology of the Cantor trees $T_i$ are suitably edited. We also expect that a construction similar to the one above can be used to embed $BS(m,n)$ into $\mapS$ when $m\neq 1$.  However, the lack of a normal form for elements in $BS(m,n)$ significantly increases the complexity of the proof.

\section{Indicable groups}\label{sec:indicable}

In this section, we give a general construction for embedding any indicable group which arises as a subgroup of a mapping class group of a surface with one boundary component into a big mapping class group in countably many non-conjugate and intrinsically infinite-type ways. 
We will need the following lemma in our construction.

\begin{lemma}\label{L:indicable}
A group $G$ is indicable if and only if there exists a presentation $G = \langle T \mid R \rangle$ such that for each $r \in R$, the total exponent sum of $r$ with respect to the generators $T$ is zero. 
\end{lemma}

Before presenting the proof of the lemma, we give an example that motivates the argument. Consider the Baumslag-Solitar group $BS(1,n)$ with its standard presentation $BS(1,n) = \langle a ,t  \mid t a t^{-1}a^{-n}\rangle$. This presentation does not have the desired property since the total exponent sum of the relator in the generators $a$ and $t$ is $1-n$. However, there exists a homomorphism $f \colon BS(1,n) \to \Z$ defined by letting $f(a) = 0$ and $f(t) =1$, so Lemma~\ref{L:indicable} tells us that there must be a presentation of $BS(1,n)$ with the desired property. If we augment the generator $a$ to be $at$ instead, then \[BS(1,n) = \left\langle at, t \,\middle\vert\, (t\cdot at\cdot t^{-1}\cdot t^{-1}) \cdot \underbrace{t(at)^{-1} \cdots t(at)^{-1}}_{n\text{ times}} \right\rangle,\] and the relator has zero total exponent sum in the generators $at$ and $t$. In this presentation, the generators of $BS(1,n)$ both map to $1$ under the homomorphism $f$, and we will use this property in the proof of the lemma. 

\begin{proof}[Proof of \Cref{L:indicable}]
Given a group $G = \langle T \, | \, R \rangle$ with all relators having total exponent sum zero, there is a well-defined homomorphism $f \colon G \to \Z$ defined by sending each generator to $1 \in \Z$.

For the other direction, assume there exists a homomorphism $f\colon G \to \Z$, and let $N = \ker(f)$. Let $N = \langle V \mid W \rangle$ be a presentation for $N$, and let $a \in G$ be such that $f(a) = 1$. Then since $G/N \cong \Z$, the group $G$ is generated by $T' = \{a\} \cup V$. If we augment the generators in $V\subseteq T'$ by $a$, then $T = \{a\} \cup \{av : v \in V\}$ is also a generating set for $G$. Importantly, the image of every one of these generators under $f$ is $1 \in \Z$.

Let $G = \langle T \mid R \rangle$ be the presentation of $G$ for the generating set $T$. If $r \in R$ is a relator, then $r$ is a word in $\langle T \rangle$ that is the identity in $G$. Thus, $f(r) = 0$, and given that every element of $T$ maps to $1 \in \Z$, the total exponent sum of $r$ with respect to $T$ must be zero. Therefore, $\langle T \mid R \rangle$ is one such desired presentation for $G$. 
\end{proof}

We can now begin our construction. Take any indicable group $G$ that arises as a subgroup of $\map(\Pi)$, where $\Pi$ is a surface with exactly one boundary component. Let $h$ be a shift map on an infinite-type surface $S$ whose domain is an embedded copy of $D_\Pi$ in $S$. As discussed in \Cref{sec:constructions}, this includes a wide range of surfaces, including surfaces $S_\Gamma(\Pi)$ built from any graph with countable vertex set that contains a  biinfinite path. 

The most trivial way to embed $G$ into $\mapS$ is to let $G$ act on one copy of $\Pi$ in $S$. Indexing the copies of $\Pi$ in $D_{\Pi}$ by $\Z$ and taking any subset of $I$ of $\Z$, $G$ can also act simultaneously on the subsurfaces $\Pi_i$ of $S$ for $i \in I$. Varying over all subsets of $\Z$ gives an uncountable collection of  copies of $G$ in $\mapS$. Unlike these embeddings, the construction in the next theorem produces an uncountable collection of copies of $G$ which do not lie in the isometry group of $S$, even if $G$ lies in the isometry group of $\Pi$, and do not lie in $\overline\mapcS$, even if $\Pi$ is compact. See \Cref{def:familyb} for the definition of a distinguished surface and the family $\familyb(\Pi)$.

\begin{theorem}\label{T:otherGs}
Let $\Pi$ be a distinguished surface and $G\leq \map(\Pi)$ an indicable group. 
Given a surface $S\in \familyb(\Pi)$, there are countably many non-conjugate embeddings of $G$ in $\mapS$ such that no embedded copy is contained in $\overline \mapcS$ and no embedded copy is contained in the isometry group for any hyperbolic metric on $S$. 
\end{theorem} 

\begin{proof} Let $h\in \mapS$ be a shift with domain $D_{\Pi}$, and let $G$ be an indicable group. Fix a presentation $\langle T \mid R \rangle$ of $G$ such that each $r \in R$ has total exponent sum zero with respect to $T$, which  exists by \Cref{L:indicable}. Since $G$ is a subgroup of $\map(\Pi)$, $G$ acts by homeomorphisms on each $\Pi_i$ in $D_\Pi$. For each $g \in G$, let $\bar g \in \mapS$ be the element that  acts as $g$ simultaneously on each $\Pi_i$ in $D_\Pi$. We claim that the group generated by $\overline T = \{\bar t h: t \in T\}$ in $\mapS$ is isomorphic to $G$. Let $\phi\colon F_T  \to \langle \overline T \rangle\leq \mapS$ be the surjective map defined by $t \mapsto  \bar t h$ for all $t \in T$, where $F_T$ is the free group on the generators $T$. We claim a word is in the kernel of this map if and only if it represents a trivial element in $G$.

Notice that $h$ and $\bar t$  commute as elements of $\mapS$, so  for any word $w\in F_T$ with total exponent sum $k\in \Z$, the image $\phi(w)$ can be written as  $\bar w h^k$. Thus, $\phi(w)$ acts trivially on $S$ if and only if $k=0$ and $\bar w$ acts trivially on each copy of $\Pi$ in $S$. The only elements $w$ with this property are those that are trivial in $G$, and elements that are trivial in $G$ have this property since products of conjugates of relators $r \in R$ have total exponent sum zero. Thus, the group $G'$ generated by $\overline T$ in $\mapS$ is isomorphic to $G$.

Any element of $G'$ that does not have total exponent sum zero with respect to $\overline T$ is not in $\overline \mapcS$, since it must shift the surfaces $\Pi_i$.   Remove finitely many copies of $\Pi$ from the domain of $h$ to obtain a new shift $h''$, and construct the group $G''=\langle \bar t h''\mid t\in T\rangle \leq \mapS$.  This group $G''$ is isomorphic to $G$ for the same reason that $G'\cong G$.  By the  same reasoning as in \Cref{sec:non-conjugate}, the complements of the supports of $G'$ and $G’'$ are non-homeomorphic.  In particular, $G'$ and $G’'$ are not conjugate.  As in \Cref{lem:nonconjugate}, this procedure produces countably many non-conjugate embeddings of $G$ into $\mapS$.  Finally, no such embedding is contained in $\overline{\mapcS}$ by construction. 
\end{proof}

It was suggested to the authors by Mladen Bestvina that one can get around constructing the presentation in \Cref{L:indicable} for the indicable group $G$ by working instead with the wreath product construction in \Cref{prop:GwrHgeneral}. More specifically, let $f\colon G \rightarrow \Z$ be a surjection from the indicable group to $\Z$. Let $\Pi$ be a surface with exactly one boundary component such that $G$ arises as a subgroup of $\map(\Pi)$, and let $S$ be a surface which admits a shift $h$ with domain $D_\Pi$. For $g\in G$, let $\bar g$ be the element which acts as $g$ on each $\Pi_i$. Then, for $g\in G$, define a new map $\psi\colon G\rightarrow G \wr \Z \leq \mapS$ via $g\mapsto \bar{g} h^{f(g)}$. 
One readily checks that this map is an injective homomorphism by observing that the restriction of the image of $G$ to $\bigoplus_{-\infty}^{\infty}G$ is the diagonal subgroup, and so the action of $\Z$ is trivial. The embedding in the proof of \Cref{T:otherGs} is exactly this map.

\Cref{T:otherGs} applies to all subgroups constructed in \Cref{sec:constructions}.  Another interesting class of examples produces embeddings of pure mapping class groups into a full mapping class group that are not induced by embeddings of the underlying surfaces.

The following corollary is immediate from \Cref{T:otherGs} and  work of Aramayona, Patel, and Vlamis \cite[Corollary 6]{APV1}, which shows that the \emph{pure} mapping class group of any surface with at least two nonplanar ends is indicable. 

\begin{corollary}\label{cor:PureEmbedding}
Let $\Pi$ be an infinite-type surface with at least two nonplanar ends and exactly one boundary component. Given any surface $S$ that admits a shift whose domain is $D_\Pi$, there exist uncountably many embeddings of $\pmap(\Pi)$ into $\mapS$ that are not induced by an embedding of $\Pi$ into $S$. In addition, none of these embeddings preserve the notion of being compactly supported. When $\Pi$ is a distinguished surface and $S \in \familyb(\Pi)$, countably many of these embeddings are non-conjugate.
\end{corollary}

\Cref{cor:PureEmbedding} is in line with a body of work aiming to find interesting homomorphisms between big mapping class groups. It also gives a natural set of examples of uncountable groups $G$ to which one can apply \Cref{T:otherGs}. We note that determining which \emph{full} mapping class groups are indicable is an important open question for both finite- and infinite-type surfaces. We now give a few examples of indicable big mapping class groups.

 \begin{exs}\label{ex:indicable}
Mann and Rafi build continuous homomorphisms from finite-index subgroups of mapping class groups to $\Z^k$ and to $\Z$ in the proofs of \cite[Lemma 6.7 \& Theorem 1.7]{MannRafi}, respectively. To find surfaces whose full mapping class groups are indicable, we focus on the cases where the subgroup has index $1$, a few of which we list below. We will define the homomorphism to $\Z$ explicitly for example (1); the others are defined similarly.
 \begin{enumerate}
     \item Let $\Pi$ be the surface with infinite genus whose end space is homeomorphic to the two-point compactification of $\Z$, that is, $E(\Pi)=\{-\infty\}\cup \Z \cup \{\infty\}$, where $E^g(\Pi)=\{\infty\}$. Let $A\subset E(\Pi)$ be the subset of ends corresponding to $-\N$, and let $B$ be the subset of ends corresponding to $\{0\}\cup \N$. This surface is colloquially called the bi-infinite flute with one end accumulated by genus, and it admits a shift with domain $D_{\Sigma}$ for a punctured disk $\Sigma$. A homomorphism $\ell\colon \map(\Pi) \arr \Z$ can be defined by \[\ell(\phi)=\left\lvert \{x \in E \mid x \in A,~\phi(x)\in B\} \right\rvert - \left\lvert \{x \in E \mid x \in B,~\phi(x)\in A\} \right\rvert. \] The map $\ell$ counts the difference in the number of punctures mapped from negative to positive and punctures mapped from positive to negative. Note that the  shift map mentioned above evaluates to $1$ under $\ell$, so the map $\ell$ is surjective.
     
     \item Let $\Pi$ be a surface of any genus whose end space consists of a Cantor set and $\{-\infty\} \cup \Z \cup \{\infty\}$, equipped with the same topology as in part (1), where the end $\{\infty\}$ is identified with a point in the Cantor set. The ends corresponding to $\{-\infty\} \cup \Z \cup \{\infty\}$ must all be planar or all nonplanar; the other Cantor set of ends can be planar or not. The homomorphism to $\Z$ is defined as above, with sets $A=-\N$ and $B=\{0\}\cup \N$. 
     \item Let $\Pi$ be the surface with infinite genus and end space $\N \cup \{\infty\}$, where only the ends corresponding to $1$ and $\infty$ are nonplanar. This surface can be visualized as the ladder surface with punctures accumulating to one end. Here we can similarly define a homomorphism to $\Z$, which instead counts the number of genus that are moved across a simple closed curve separating the ends in $E^G$.
 \end{enumerate}
 \end{exs}

 The common thread in the examples above is that the two ends of the shift map are of different topological types so that no element of $\map(\Pi)$ can exchange the two ends. This is the key fact necessary to ensure that the map $\ell$ above is a well-defined homomorphism of $\map(\Pi)$ and not of a proper subgroup of $\map(\Pi)$.
  
 Each of the examples above can be modified to have exactly one boundary component.  The third  example can be extended to uncountably many more examples by replacing one of the isolated planar ends with a disk punctured by any closed subset of the Cantor set, of which there are uncountably many.   
 
 Moreover, in each case, $\Pi$ is like a distinguished surface in the sense that if $S-D_\Pi$ has finitely many nonplanar ends in Cases (1) and (3), then $\mapS$ can be used as the input for \Cref{T:otherGs}. In Case (2), if $S-D_\Pi$ has finitely many nonplanar (resp. planar) ends when the ends of $E(\Pi)$ corresponding to $\{-\infty\} \cup \Z \cup \{\infty\}$ are nonplanar (resp. planar), then $\mapS$ can be used as the input for \Cref{T:otherGs}.  Therefore, we can construct countably many non-conjugate embeddings of $\map(\Pi)$ into $\mapS$ in all such cases.

\section{Combination Theorem}\label{sec:combo}
 In this section, we give a construction that takes as its input a set of indicable subgroups of mapping class groups of surfaces with one boundary component and outputs a new surface whose mapping class group contains a new indicable subgroup, called the $\star$-product, of intrinsically infinite type built from the original subgroups.

\begin{definition}\label{def:StarProduct}
Given two subgroups $H_1$ and $H_2$ of groups $G_1$ and $G_2$, respectively, let  \[(G_1, H_1) \star (G_2, H_2) := (G_1 \ast G_2) / \llangle [G_1, H_2], [H_1,G_2]\rrangle.\] More generally, given $G_1, \ldots, G_n$ with  subgroups $H_1, \ldots, H_n$, let \[(G_1, H_1) \star \cdots \star (G_n, H_n) := G_1 \ast \cdots \ast G_n / \llangle [G_i, H_j]: i \neq j \rrangle.\]
\end{definition}
These groups are an interpolation between free products (where the $H_i$ are trivial) and direct products (where $H_i=G_i$ for all $i$). Examples include some right-angled Artin groups and certain graph products. We are interested in the case where $G_i$ are indicable groups and the $H_i$ are the kernels of the surjections to $\Z$.

\begin{lemma}
Let $G_1, \dots, G_n$ be indicable groups with surjective maps $f_i\colon G_i \to \Z$, and let $H_i = \ker(f_i)$. Then the group $(G_1,H_1) \star \cdots \star (G_n,H_n)$ is also indicable.
\end{lemma}

\begin{proof}
Let $T_i$ be a generating set for $G_i$.  Then there is a map $\phi\colon (G_1,H_1) \star \cdots \star (G_n,H_n)\to G_1$ defined by $\phi(t)=1$ for each $t\in T_i$ with $i\neq 1$, and $\phi(t')=t'$ for each $t'\in T_1$. Here $1$ is the identity element of $G_1$.  This map $\phi$ is a homomorphism which restricts to the identity on $G_1$.  By post-composing $\phi$ with $f_1$, we obtain the desired map $(G_1,H_1) \star \cdots \star (G_n,H_n) \to \Z$.
\end{proof}

\begin{figure}
    \centering
    \def\svgwidth{5in}
    \import{images/}{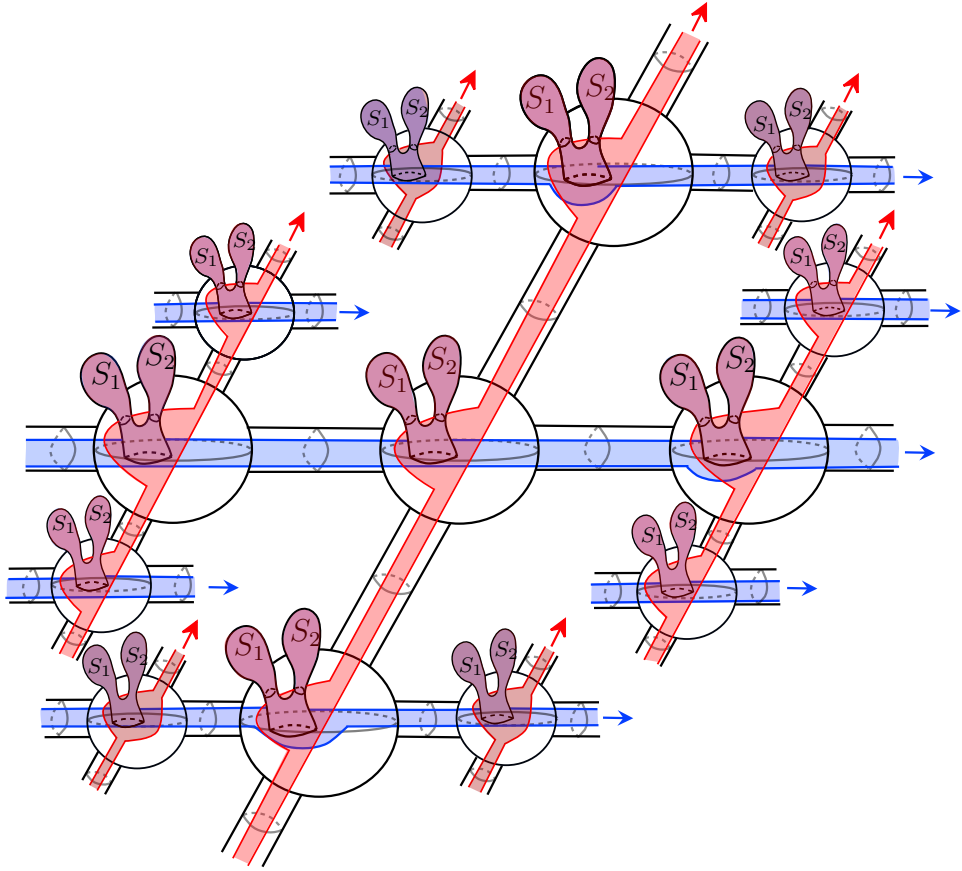}
    \caption{The domains of the two multipushes $x_a$ (blue) and $x_b$ (red) in the proof of \Cref{cor:otherS} in the case that $\Gamma$ is the Cayley graph of the free group generated by $a$ and $b$.}
    \label{fig:multipushes}
\end{figure}

We are now ready to prove our main combination theorem, of which \Cref{thm:main} is a special case.  

\begin{theorem}\label{cor:otherS}
For $i=1,\dots, n$, let $S_i$ be a distinguished surface, and let $\Pi$ be obtained from $\#_n S_i$ by capping off $n-1$ boundary components. Let $S$ be a Schreier surface in $\familyb(\Pi)$ for a triple $(G,T,H)$ with $|T|=n$.  For each $i$, let $G_i$ be an indicable group that embeds in $\map(S_i)$,  fix a surjective map $f_i\colon G_i\to \Z$, and let $H_i=\ker f_i$.  
 
 Then, there are countably many non-conjugate embeddings of the indicable group 
 \[
 G_1*G_2*\cdots * G_n /\llangle \bigcup_{i\neq j} [G_i,H_j] \rrangle
 \]
 into  $\mapS$, none of which lie in $\overline\mapcS$.
\end{theorem}

\begin{proof}
We prove the theorem for $n=2$ for simplicity of notation, but the same proof works for all $n$. 
Let $a$ and $b$ be two distinct generators of the group $G$ for the Schreier surface $S$. By construction, $S$ admits two multipushes $\mpush_a$ and $\mpush_b$, where each acts as simultaneous pushes, as in \Cref{D:mpush}.  See \Cref{fig:multipushes}.

By \Cref{L:indicable}, each surjection $f_i \colon G_i \to \Z$ gives rise to a presentation $G_i = \langle T_i \mid R_i \rangle$  such that  every $r\in R_i$ has total exponent sum zero with respect to $T_i$ for $i = 1, 2$. Note that, under this map to $\Z$, the image of each generator is $1$. As this is a homomorphism, the map $f_i$ sends a group element of $G_i$ to its total exponent sum with respect to $T_i$, and this is an invariant of the group element.

Similarly to \Cref{T:otherGs}, for each $g \in G_i$, define an element $\bar g \in \mapS$, where $\bar g$ acts as $g$ simultaneously on each copy of $\Pi$ in the domains of $\mpush_a$ and $\mpush_b$ in $S$. Note that since the copies of $S_1$ and $S_2$ in each copy of $\Pi$ are disjoint, the elements $\bar g_1$ and $\bar g_2$ of $\mapS$ commute for any $g_1 \in G_1$ and $g_2 \in G_2$. This key fact will be used several times throughout the proof. Letting $\widetilde T_1 = \{\bar t \mpush_a: t \in T_1\}$ and  $\widetilde T_2 = \{\bar t \mpush_b: t \in T_2\}$, we claim that the group generated by $\widetilde T_1 \cup \widetilde T_2$ in $\mapS$ is isomorphic to $G_1*G_2 / \llangle [G_1, H_2] \cup [H_1, G_2] \rrangle$.

First, we claim that the group generated by $\widetilde T_1$ in $\mapS$ is isomorphic to $G_1$ under the map $t \to \bar t x_a$. Observe that for any $t\in T_1$, the element $\bar{t}$ commutes with $\mpush_a$, and in fact also with $\mpush_b$, although we will only need the first fact now. Thus, the image of a word $w \in \langle T_1\rangle$ under this map can be expressed as the product of the corresponding word in $\{\bar t : t\in T_1\}$ and $\mpush_a^n$ where $n$ is the exponent sum of the word $w$ with respect to $T_1$. For $n \neq 0$, the support of $\mpush_a^n$ is not contained in the vertex surfaces, so we observe that a word in $\widetilde T_1$ is the trivial element in $\mapS$ if and only if both the corresponding word in $\{\bar t : t\in T_1\}$ acts trivially on $S$ and $n=0$. However, since $\langle \bar t : t\in T_1\rangle$ is canonically isomorphic to $G_1=\langle t : t\in T_1\rangle$ and the presentation for $G_1$ was chosen so that every relator has total exponent sum zero with respect to $T_1$, the second condition is redundant. Thus, a word on $\widetilde{T_1}$ is the trivial element in $\mapS$ iff the corresponding word in $G_1$ is trivial. Therefore, the group generated by $\widetilde T_1$ and $G_1$ are isomorphic, and a symmetric argument shows that the group generated by $\widetilde T_2$ is isomorphic to $G_2$. Setting up notation for the rest of the proof, denote by $\widetilde G_1$ and $\widetilde G_2$, respectively, these isomorphic copies of $G_1$ and $G_2$. That is to say $\widetilde G_i = \langle \widetilde T_i \rangle$ as subgroups of $\mapS$ for $i=1,2$. Moreover, let $\widetilde H_i$ be the subgroup of $\widetilde G_i$ that is canonically isomorphic to $H_i$, for $i = 1,2$.

We now have a surjective map  
\[\phi\colon G_1*G_2 \rightarrow \langle \widetilde T_1 \cup \widetilde T_2\rangle,\]
where $\langle \widetilde T_1 \cup \widetilde T_2\rangle = \langle \widetilde G_1 \cup \widetilde G_2\rangle$, and it suffices to show that the kernel of $\phi$ is normally generated by the commutators $[G_1, H_2]$ and $[H_1, G_2].$

For this, we first show that every element of $[G_1, H_2]$ and $[H_1, G_2]$ is in the kernel. Let $w=[g, h]$ where $g\in G_1$ and $h \in H_2$. As $h \in H_2$, it has exponent sum $0$ and so $\phi(h)=\bar h$. Let $m$ be the exponent sum of $g$. As noted above, for any $t\in T_1\cup T_2$, the element $\bar{t}$ commutes with both $\mpush_a$ and $\mpush_b$, and so it is also true that $\bar h$ and $\bar g$ commute with both $\mpush_a$ and $\mpush_b$. Finally, observe that as $\bar g$ and $\bar h$ have disjoint supports, they also commute. 
We now evaluate $\phi(w):$ 

\begin{align*}
\phi(w)&= \phi(g^{-1}h^{-1}gh) \\
&=\bar g^{-1} \mpush_a^{-m} \bar h^{-1} \bar g \mpush_a^m \bar h  \\
&= \mpush_a^{-m}\mpush_a^m [\bar g, \bar h]\\
&=1.
\end{align*}
An analogous argument shows $[H_1, G_2]$ is in the kernel of $\phi$. Therefore, the induced map 

\[\bar\phi:G_1*G_2 /  \llangle [G_1, H_2] \cup [H_1, G_2] \rrangle \rightarrow \langle \widetilde T_1 \cup \widetilde T_2\rangle,\]
 is  surjective. 
 
 To complete the proof, we must show that $\bar\phi$ is injective, as well. We view elements of $[G_1,H_2]\cup [H_1,G_2]$ as relators in $G_1*G_2 /  \llangle [G_1, H_2] \cup [H_1, G_2] \rrangle$ and note that applying a relator in $[G_1,H_2]\cup [H_1,G_2]$ to the word $w$ in $G_1*G_2$ amounts to applying the corresponding element of $[\widetilde G_1, \widetilde H_2]\cup [\widetilde H_1,\widetilde G_2]$ to the corresponding word $\widetilde{w}$ in $\widetilde G_1 * \widetilde G_2$. Thus, showing that $\bar\phi$ is injective is equivalent to showing that if $w \in G_1* G_2$ maps to the trivial element in $\mapS$, then $\widetilde{w}$ can be reduced to the identity via free reductions and repeated applications of relators in the set $[\widetilde G_1, \widetilde H_2]\cup[\widetilde H_1, \widetilde G_2]$.

We now proceed with the proof, which will be via strong induction on the length of $w$ as an alternating word in $G_1$ and $G_2$. Recall that an element $w= g_1g_2\dots g_n \in G_1*G_2$ is in normal form if each $g_j$ is in either $G_1\backslash \{1\}$ or $G_2\backslash \{1\}$ and if $g_j \in G_i$ then $g_{j+1}\in G_{i+1}$, where the indices $i$ are taken modulo 2.

We establish two base cases. Assume $n=1$. Then $w$ is in either $G_1$ or $G_2$. As $\phi$ restricted to either $G_1$ or $G_2$ is injective, $w$ must be the trivial element, so no reductions are needed and the $n=1$ case is vacuous.

Similarly, if $n=2$, then $w$ is of the form $g_1g_2$, and, without loss of generality, we can assume $g_1\in G_1$ and $g_2\in G_2$. In this case, the $\widetilde{w}$ is
$\mpush_a^k \bar g_1 \mpush_b^\ell \bar g_2$ for some values of $k$ and $\ell$, which is equal to $\mpush_a^k \mpush_b^\ell \bar g_1  \bar g_2$ in $\mapS$. Again, as $\mpush_a$ and $\mpush_b$ generate a free group for which nontrivial elements have support not contained in the vertex surfaces, it must be that $k=\ell=0$. Moreover, as $\bar g_1$ and $\bar g_2$ act with disjoint supports, they too must be the identity. Thus, we conclude that $g_1=g_2=1$, and this case is also vacuous. 

Let, $n\geq 3$, and suppose $w=g_1 \cdots g_n$. Again, without loss of generality, we assume $g_1\in G_1$. We will also assume $n$ is odd, but one checks that if $n$ is even, the argument differs only in notation. In this case,   
\[
\widetilde{w}=\mpush_a^{k_1} \bar g_1 \mpush_b^{k_2} \bar g_2\mpush_a^{k_3} \bar g_3 \cdots \mpush_a^{k_{n}} \bar g_{n}\]
which equals \[
\mpush_a^{k_1}\mpush_b^{k_2}\cdots\mpush_a^{k_{n}}\bar g_1 \bar g_2 \cdots \bar g_{n}\] as an element in $\mapS$.

In order to act trivially, the product $\mpush_a^{k_1}\mpush_b^{k_2}\cdots\mpush_a^{k_{n}}$ must freely reduce the to identity, and as the product alternates, this implies that there is some $j$ such that $k_j=0$. In turn, this implies that $g_j$ is in $H_1$ or $H_2$. 

Suppose $j=1$ so that $\mpush_a^{k_1} \bar g_1=\mpush_a^0 \bar g_1 \in \widetilde H_1$. Then an application of a relator in $[\widetilde H_1, \widetilde G_2]$ allows us to move the first term $\mpush_a^{k_1} \bar g_1=\mpush_a^0 \bar g_1$ across the second term $\mpush_b^{k_2} \bar g_2$, and we can shorten the original normal form of $\widetilde{w}$ to 
\begin{align*}
&\mpush_b^{k_2} \bar g_2 (\mpush_a^{k_1} \bar g_1 \mpush_a^{k_3} \bar g_3) \cdots \mpush_a^{k_{n-1}} \bar g_{n-1} \mpush_b^{k_n} \bar g_n\\
=&\mpush_b^{k_2} \bar g_2 (\mpush_a^{k_1+k_3} \bar g_1 \bar g_3) \cdots \mpush_a^{k_{n-1}} \bar g_{n-1} \mpush_b^{k_n} \bar g_n.
\end{align*}

If this new word is in normal form in the free product $\widetilde G_1* \widetilde G_2$, then by induction, $\widetilde{w}$ can be reduced to the identity and we are done. On the other hand, if it is not in normal form, which is precisely when $\mpush_a^{k_1+k_3} \bar g_1 \bar g_3=1$, then we apply free reductions until it is. After there are no more free reductions to apply, the word is either already trivial or in normal form of shorter length. Once again, the inductive hypothesis implies that the word can be reduced to the identity via free reductions and repeated applications of relators in the set $[\widetilde G_1, \widetilde H_2]\cup[\widetilde H_1, \widetilde G_2]$. This completes the induction step for the case $j=1$. 

It remains to check the cases where $1< j \leq n.$ When $j=n$, the proof follows as above with minor notational differences. When $1<j<n$, the proof again follows as above with the exception that when $n\geq 4$, the word length will drop by two immediately before free reductions. Combining all cases, we see that $\bar \phi$ is injective and the group generated by $\widetilde T_1 \cup \widetilde T_2$ in $\mapS$ is isomorphic to $G_1*G_2 / \llangle [G_1, H_2] \cup [H_1, G_2] \rrangle$.

  We have shown that $(G_1, H_1) \star (G_2, H_2)$ embeds in $\mapS$.  As in the proof of \Cref{thm:indicable}, by removing finitely many copies of $\Pi$ from the domains of $x_a$ and $x_b$, we obtain countably many non-conjugate embeddings of $(G_1, H_1) \star (G_2, H_2)$ into $\mapS$.  By construction, no such embedding is contained in $\overline\mapcS$.
\end{proof}

\subsection{Applications of \Cref{thm:main}}
\label{sub:RAAGs}

\Cref{thm:main} produces embeddings of $\star$-products into $\mapS$.  In general, $(G_1,H_1)\star (G_2,H_2)$ need not be finitely presented, even when the groups $G_i$ are finitely presented.
For example, consider the indicable group $\mathbb F_2=\<a,b\>$ with the map to $\Z$ defined by $a\mapsto 1$ and $b\mapsto 0$. It is an exercise to see that the kernel $K$ of this map is not finitely generated; see Exercise 7 of Section 1.A in \cite{hat}. Therefore, if $G_1=G_2=\mathbb F_2$ and $H_1=H_2=K$, then $(G_1,H_1)\star (G_2,H_2)$ is a finitely generated but infinitely presented group.
However, there are  instances where the $\star$-product is a recognizable finitely presented group.   

\begin{example}
As a first example, consider a collection of groups $H_1,\dots, H_n$ which are subgroups of $\map(S_i)$ where $S_i$ is a surface with one compact boundary component. Then after possibly increasing the complexity of $S_i$, the groups $G_i=\Z \times H_i$ are also subgroups of $\map(S_i)$. 
The groups $H_i$ are the kernels of the projection maps of $G_i$ onto the $\Z$ factor. As such, the following group is the associated $\star$-product, which embeds in $\mapS$ for $S = S_\Gamma(\Pi)$ by \Cref{thm:main}: \[(G_1,H_1)\star \dots \star (G_n,H_n)\simeq H_1 \times \dots \times H_n \times F_n.\]
\end{example}

For a second class of examples, consider the case where $G$ is a right-angled Artin group defined by the finite graph $\G$. By work of Clay, Leininger, Mangahas \cite{CLM12} and Koberda \cite{Koberda}, every right-angled Artin group embeds as a subgroup of the mapping class group of some finite-type surface of sufficient complexity.

The map $\chi\colon G\arr \Z$ which sends each generator of $G$ from $\Gamma$ to $1$ is a homomorphism, so  $G$ is indicable. The kernel of this map is called the \textit{Bestvina-Brady group} defined on $\G$, denoted by $BB_{\G}$. 
 Bestvina-Brady groups are a vast and varied class of groups. There is a large body of work connecting combinatorial conditions on $\G$ to algebraic properties of $BB_{\G}$.
First, the group $BB_{\G}$ is finitely generated exactly when $\G$ is connected. Further, $BB_{\G}$ is finitely presented exactly when the flag complex associated to $\G$ is simply connected. See \cite{BB} for more details.

For finitely presented $BB_{\G}$, Dicks and Leary compute an explicit presentation for $BB_{\G}$ in \cite{DL99}, which can be expressed in terms of the generators of $A_{\G}$. 
In the particular case when $\G$ is any tree on $n$ vertices, this presentation can be used to show that the Bestvina-Brady group is $F_{n-1}$. The following example makes use of this fact.

\begin{example}\label{ex:BB}
    Let $P_4$ represent the path graph on $4$ vertices  $a,b,c,d$ as in \Cref{fig:PathGraph}. Then the right-angled Artin group $A_{P_4}$ has presentation $\<a,b,c,d \mid [a,b], [b,c], [c,d]\>$. The associated group $BB_{\G}$ is the free group generated by $\<ab\inv, bc\inv, cd\inv\>$. 
    
    \begin{figure}[h]
    \begin{tikzpicture}[scale=1, every node/.style={circle, draw, minimum size=8mm}]
    \node (a) at (0,0) {a};
    \node (b) at (2,0) {b};
    \node (c) at (4,0) {c};
    \node (d) at (6,0) {d};

    \draw (a) -- (b) -- (c) -- (d);
\end{tikzpicture}
\caption{The Path graph}\label{fig:PathGraph}
\end{figure}

Letting $G_1=G_2=A_{P_4}$, with $H_1\simeq H_2 \simeq F_3$, one can obtain an explicit presentation for $(A_{P_4},F_3)\star (A_{P_4},F_3)$ following the work of Dicks and Leary. \Cref{thm:main}  then implies that $(A_{P_4},F_3)\star (A_{P_4},F_3)$ embeds into $\mapS$ with $S=S_\Gamma(\Pi)$ for a sufficiently complicated $\Pi$. In particular, $(A_{P_4},F_3)\star (A_{P_4},F_3)$ embeds into $\mapS$ for $S \in \familya_\infty$ (see \Cref{def:familya}). An example of such a surface $S$ is the blooming Cantor tree surface (the boundaryless surface whose end space is a Cantor set of nonplanar ends). More generally, for any right-angled Artin groups $A_{\G_1}$, $A_{\G_2}$, the group $(A_{\G_1},BB_{\G_1})\star (A_{\G_2},BB_{\G_2})$ embeds in $\mapS$ for $S \in \familya_\infty$. 

    \end{example}

 In addition, Chang and Ruffoni \cite{CR} solve the right-angled Artin group recognition problem for Bestvina-Brady groups, i.e., they give sufficient and necessary conditions on $\G$ that determine if $BB_{\G}$ is itself a right-angled Artin group. In particular, they show the following.

\begin{theorem}[Chang-Ruffoni, \cite{CR}]
     If $\G$ admits a tree 2-spanner $T$, then $BB_{\G}$ is a right-angled Artin group. More precisely, the Dicks–Leary presentation can be simplified to the standard right-angled Artin group presentation with generating set $E(T)$. Moreover, we have $BB_{\G}=A_{T^*}$.
\end{theorem}
\noindent Here, $E(T)$ denotes the edge set of $T$, $T^*$ is the dual graph to $T$, and $A_{T^*}$ is the right-angled Artin group with defining graph $T^*$.
A tree $2$-spanner of $\G$ is
a spanning tree $T$ of $\G$ such that for all $x,y\in V(T)$, we have $d_T(x,y)\leq 2 d_{\G}(x,y)$.

One consequence of their work is that the class of Bestvina-Brady groups includes all right-angled Artin groups. Therefore, letting $A_1$ and $A_2$  be any right-angled Artin groups, their work implies that $\Gamma_1$ and $\Gamma_2$ can be chosen so that the corresponding Bestvina-Brady groups are precisely $A_1$ and $A_2$. It follows from \Cref{thm:main} that the group $(A_{\Gamma_{1}}, A_1)\star (A_{\Gamma_{2}}, A_2)$ embeds in $\mapS$ for $S \in \familya_\infty$.

\bibliographystyle{alpha}

\begin{thebibliography}{}

\end{thebibliography}


\begin{thebibliography}{ACCL21}

\bibitem[ACCL21]{ACCL20}
Afton, Santana; Calegari, Danny; Chen, Lvzhou; Lyman, Rylee~Alanza.
\newblock Nielsen realization for infinite-type surfaces.
\newblock {\em Proc. Amer. Math. Soc.}, 149(4):1791--1799, 2021. 
\mrev{4242332}, \zbl{1468.57011}.
 

\bibitem[AF21]{AF21}
Aramayona, Javier; Funar, Louis.
\newblock Asymptotic mapping class groups of closed surfaces punctured along
  {C}antor sets.
\newblock {\em Mosc. Math. J.}, 21(1):1--29, 2021. 
\mrev{4219034}, \zbl{1475.57031}.

\bibitem[AIM]{AIM-PL}
{AIM} {P}roblem {L}ist: {S}urfaces of infinite type.
\newblock \texttt{available at http://aimpl.org/genusinfinity}.


\bibitem[All06]{All06}
Allcock, Daniel.
\newblock Hyperbolic surfaces with prescribed infinite symmetry groups.
\newblock {\em Proc. Amer. Math. Soc.}, 134(10):3057--3059, 2006.
\mrev{2231632}, \zbl{1094.51004}.



\bibitem[ALM24]{ALM}
Aramayona, Javier;  Leininger, Christopher~J.; McLeay, Alan.
\newblock Big mapping class groups and the co-{H}opfian property.
\newblock {\em Michigan Math. J.}, 74(2):253--281, 2024.
\mrev{4739839}, \zbl{1540.57015}.


\bibitem[ALS09]{AramayonaLeiningerSouto}
Aramayona, Javier; Leininger, Christopher~J.; Souto, Juan.
\newblock Injections of mapping class groups.
\newblock {\em Geometry \& {T}opology}, 13(5):2523--2541, 2009.
\mrev{2529941}, \zbl{1225.57001}.

\bibitem[AMP25]{AMP21}
Abbott, Carolyn; Miller, Nicholas; Patel, Priyam.
\newblock Infinite-type loxodromic isometries of the relative arc graph.
\newblock {\em Algebr. Geom. Topol.}, 25(1):563--644, 2025.
\mrev{4885571}, \zbl{08027956}.


\bibitem[APV20]{APV1}
Aramayona, Javier; Patel, Priyam;  Vlamis, Nicholas~G.
\newblock The first integral cohomology of pure mapping class groups.
\newblock {\em Int. Math. Res. Not. IMRN}, (22):8973--8996, 2020.
\mrev{4216709}, \zbl{1462.57020}.


\bibitem[APV21]{APV2}
Aougab, Tarik; Patel, Priyam; Vlamis, Nicholas~G.
\newblock Isometry groups of infinite-genus hyperbolic surfaces.
\newblock {\em Math. Ann.}, 381(1-2):459--498, 2021.
\mrev{4322618}, \zbl{1476.57041}.


\bibitem[AS13]{AramayonaSouto}
Aramayona, Javier; Souto, Juan.
\newblock Homomorphisms between mapping class groups.
\newblock {\em Geometry \& {T}opology}, 16(4):2285--2341, 2013.
\mrev{3033518}, \zbl{1262.57003}.

\bibitem[BB97]{BB}
Bestvina, Mladen; Brady, Noel.
\newblock Morse theory and finiteness properties of groups.
\newblock {\em Invent. Math.}, 129(3):445--470, 1997.
\mrev{3033518}, \zbl{1262.57003}.



\bibitem[BDR20]{BVD}
Bavard, Juliette; Dowdall, Spencer; Rafi, Kasra.
\newblock Isomorphisms between big mapping class groups.
\newblock {\em Int. Math. Res. Not. IMRN}, (10):3084--3099, 2020.
\mrev{4098634}, \zbl{1458.57014}.


\bibitem[CLM12]{CLM12}
Clay, Matt~T.; Leininger, Christopher~J.; Mangahas, Johanna.
\newblock The geometry of right-angled {A}rtin subgroups of mapping class groups.
\newblock {\em Groups Geom. Dyn.}, 6(2):249--278, 2012.
\mrev{2914860}, \zbl{1245.57004}.


\bibitem[CR24]{CR}
Chang, Yu-Chan; Ruffoni, Lorenzo.
\newblock A graphical description of the {B}{N}{S}-invariants of {B}estvina--{B}rady
  groups and the {R}{A}{A}{G} recognition problem.
\newblock {To appear: \em Groups, Geometry, and Dynamics}, 2024.


\bibitem[DL99]{DL99}
Dicks, Warren; Leary, Ian~J.
\newblock Presentations for subgroups of {A}rtin groups.
\newblock {\em Proc. Amer. Math. Soc.}, 127(2):343--348, 1999.
\mrev{1605948}, \zbl{0923.20032}.



\bibitem[FM12]{primer}
Farb, Benson; Margalit, Dan.
\newblock {\em A primer on mapping class groups}, volume~49 of {\em Princeton
  Mathematical Series}.
\newblock Princeton University Press, Princeton, NJ, 2012.
\mrev{2850125}, \zbl{1245.57002}.



\bibitem[Gro77]{Gro77}
Gross, Jonathan~L.
\newblock Every connected regular graph of even degree is a {S}chreier coset graph.
\newblock {\em Journal of Combinatorial Theory, Series B}, 22(3):227--232,
  1977.
\mrev{0450121}, \zbl{0369.05042}.

  

\bibitem[Hat02]{hat}
Hatcher, Allen.
\newblock {\em Algebraic topology}.
\newblock Cambridge University Press, Cambridge, 2002.
\mrev{1867354}, \zbl{1044.55001}.



\bibitem[Iva84]{Ivanov}
Ivanov, Nikolai~V.
\newblock Algebraic properties of the {T}eichm\"{u}ller modular group.
\newblock {\em Dokl. Akad. Nauk SSSR}, 275(4):786--789, 1984.
\mrev{0745513}, \zbl{0586.20026}.



\bibitem[Ker23]{Kerekjarto}
Ker\'{e}kj\'{a}rt\'{o}, B\'{e}la~Von.
\newblock {\em Vorlesungen \"{u}ber Topologie. I}.
\newblock Spring, Berlin, 1923.


\bibitem[Kob12]{Koberda}
Koberda, Thomas.
\newblock Right-angled {A}rtin groups and a generalized isomorphism problem for finitely generated subgroups of mapping class groups.
\newblock {\em Geometric and Functional Analysis}, 22(6):1541--1590, 2012.
\mrev{3000498}, \zbl{1282.37024}.


\bibitem[LL20]{LanierLoving}
Lanier, Justin; Loving, Marissa.
\newblock Centers of subgroups of big mapping class groups and the {T}its alternative.
\newblock {\em Glasnik Matematicki}, 55:85--91, 06 2020.
\mrev{4115214}, \zbl{1484.20061}.


\bibitem[Lub95]{Lub95}
Lubotzky, Alexander.
\newblock Cayley graphs: eigenvalues, expanders and random walks.
\newblock In {\em Surveys in combinatorics, 1995 ({S}tirling)}, volume 218 of
  {\em London Math. Soc. Lecture Note Ser.}, pages 155--189. Cambridge Univ.
  Press, Cambridge, 1995.
\mrev{1358635}, \zbl{0835.05033}.

  

\bibitem[McC85]{McCarthy}
McCarthy, John.
\newblock A ``{T}its-alternative'' for subgroups of surface mapping class
  groups.
\newblock {\em Trans. Amer. Math. Soc.}, 291(2):583--612, 1985.
\mrev{0800253}, \zbl{0579.57006}.



\bibitem[MR23]{MannRafi}
Mann, Kathryn; Rafi, Kasra.
\newblock Large-scale geometry of big mapping class groups.
\newblock {\em Geom. Topol.}, 27(6):2237--2296, 2023.
\mrev{4634747}, \zbl{1566.57012}.


\bibitem[PM07]{Prishlyak-Mischenko2007}
Prishlyak, Alexander Olegovich; Mischenko, Kateryna.
\newblock Classification of noncompact surfaces with boundary.
\newblock {\em Methods Funct. Anal. Topology}, 13(1):62--66, 2007.
\mrev{2308580}, \zbl{1199.57010}.


\bibitem[PV18]{PatelVlamis}
Patel, Priyam; Vlamis, Nicholas~G.
\newblock Algebraic and topological properties of big mapping class groups.
\newblock {\em Algebr. Geom. Topol.}, 18(7):4109--4142, 2018.



\bibitem[Ric63]{Richards}
Richards, Ian.
\newblock On the classification of noncompact surfaces.
\newblock {\em Trans. Amer. Math. Soc.}, 106:259--269, 1963.
\mrev{0143186}, \zbl{0156.22203}.


\end{thebibliography}

\end{document}